\documentclass{amsart}
\usepackage{amsmath}
\usepackage{algorithm}
\usepackage{blkarray}
\usepackage{multicol,graphicx,color}
\usepackage{amsmath,amssymb,amsthm}
\usepackage{tikz}
\usepackage[noend]{algpseudocode}
\usepackage{subcaption, caption}
\usepackage{hyperref}
\usepackage{mathrsfs}
\usepackage{diagrams}
\usepackage{appendix}
\usepackage{placeins} 

\newcommand{\C}{\mathbb{C}}

\newcommand{\Z}{\mathbb{Z}}
\newcommand{\N}{\mathbb{Z}^{\ge 0}}
\newcommand{\I}{\mathcal{I}}
\newcommand{\quotientalgebra}{\mathcal{A}}
\newcommand{\quotientbasis}{\mathcal{B}}
\newcommand{\ringbasis}{\mathbf{B}}
\newcommand{\pring}{\mathscr{R}}

\newcommand{\x}{\mathbf{x}}
\newcommand{\z}{\mathbf{z}}

\newcommand{\numroots}{r}

\DeclareMathOperator{\rank}{rank}
\renewcommand{\root}{\mathbf{r}}
\newcommand{\norm}[1]{\left\lVert#1\right\rVert}

\newcommand{\0}{\mathbf{0}}

\newcommand{\pd}{\beta} 
\newcommand{\dMac}{d}
\newcommand{\degree}{\beta}
\newcommand{\dimension}{n}
\DeclareMathOperator{\Mac}{\textup{Mac}}
\DeclareMathOperator{\QR}{\textup{QR}}
\DeclareMathOperator{\SVD}{\textup{SVD}}
\DeclareMathOperator{\MatMult}{MM}
\DeclareMathOperator{\Backsolve}{Back}
\newcommand{\monomialsdeg}[1]{H_{#1}}
\newcommand{\monomialsleqdeg}[1]{V_{#1}}
\newcommand{\polysdeg}[1]{T_{#1}}
\newcommand{\polysinmac}[1]{S_{#1}}
\newcommand{\nullitymac}[1]{\textup{nullity}(\Mac(#1))}
\newcommand{\nullitymacappendix}[1]{\gamma_{#1}}
\newcommand{\nullSpace}[1]{N_{#1}}
\newcommand{\DBDStep}[1]{G_{#1}} 
\newcommand{\degreelim}{\lim_{\degree\rightarrow\infty}}
\newcommand{\degreelimsup}{\limsup_{\degree\rightarrow\infty}}
\newcommand{\degreeliminf}{\liminf_{\degree\rightarrow\infty}}
\newcommand{\dimensionlim}{\lim_{\dimension\rightarrow\infty}}

\newcommand{\floor}[1]{\left\lfloor #1 \right\rfloor}

\newcommand{\ceil}[1]{\left\lceil #1 \right\rceil}

\newcommand{\halfFloor}[1]{\floor{\frac{#1}{2}}}
\newcommand{\halfCeil}[1]{\ceil{\frac{#1}{2}}}

\renewcommand{\mathcal}{\mathscr}
\newcommand{\hermitian}{^H}
\newcommand{\transpose}{^\top}

\DeclareMathOperator{\nullity}{\textup{nullity}}
\newcommand{\lowdegring}{\mathbb C[x_1,\ldots,x_n;\dMac-1]}

\newtheorem{theorem}{Theorem}[section]
\newtheorem{proposition}[theorem]{Proposition} 
\newtheorem{lemma}[theorem]{Lemma} 
\newtheorem{definition}[theorem]{Definition} 

\title[Normal-Form Rootfinding]{Analysis of Normal-Form Algorithms for Solving Systems of Polynomial Equations}
\author{Suzanna Parkinson}
\thanks{Partially supported by a Mentoring Environment Grant, Brigham Young University.}
\author{Hayden Ringer}
\author{Kate Wall}
\author{Erik Parkinson}
\author{Lukas Erekson}
\author{Daniel Christensen}
\author{Tyler J. Jarvis}
\thanks{Partially supported by NSF grant DMS-1564502.}
\begin{document}
\maketitle
\markleft{Parkinson et al.}
\begin{abstract}

We examine several of the normal-form multivariate polynomial rootfinding methods of Telen, Mourrain, and Van Barel and some variants of those methods. We analyze the performance of these variants  in terms of their asymptotic temporal complexity as well as speed and accuracy on a wide range of numerical experiments.  
All variants of the algorithm are problematic for systems in which many roots are very close together. We analyze performance on one such system in detail, namely the ``devastating example'' that Noferini and Townsend used to demonstrate instability of resultant-based methods.
\end{abstract}

\section{Introduction}\label{sec:assumptions}

We are interested in efficient numerical algorithms to solve generic systems of multivariate polynomials $\{p_1,\dots, p_n\}$. That is, we wish to find the set $Z(p_1,\dots, p_n) = \{\x :\,  p_i(\x) = \0,  1 \le i \le n\}$. By the term \emph{generic} we mean that the system has only a finite number of roots, no multiple roots, and no roots at infinity; that is, 
$\I = ( p_1,\ldots,p_n )$ is a radical, zero-dimensional ideal with no zeros at infinity.

One powerful way to way to solve these systems is with eigenvalue-based methods, which are  multidimensional generalizations of companion-matrix methods.  An essential step in these methods is finding a basis for the quotient algebra $\quotientalgebra = \C[x_1,\dots,x_n]/\I$.  Telen, Mourrain, and Van Barel in \cite{Telen,Telen2,Telen4} developed several algorithms to numerically construct a basis for this quotient algebra. 

In this article we analyze several variations of their methods, including
using different matrix decompositions at key steps, and also consider some proposed speedups from \cite{Telen4}. We examine the temporal complexity and summarize the timing and accuracy (residuals) from a number of numerical experiments for each variation. 

Unfortunately, these algorithms are unstable and can perform poorly when many roots are close to each other. To examine these problem cases, we look at a system from Noferini and Townsend \cite{Noferini} that Townsend calls the \emph{devastating example}, and we discuss the inherent conditioning problems of these eigenvalue methods on such systems. Despite these issues, these algorithms perform well on systems of polynomials for which the roots are sufficiently separated from each other.

\subsection{Outline}
The basic structure of this paper is as follows. In the next section, we introduce eigenvalue methods for rootfinding and the Macaulay matrix. Sections~\ref{sec: Direct Macaulay Reduction} and~\ref{sec: Null Space Macaulay Reduciton} describe two different ways of using the Macaulay matrix to construct a basis for $\quotientalgebra$ in order to find roots. In Section~\ref{sec: Speedups}, we describe potential speedups to the previous methods. We describe the temporal complexity of each algorithm (both methods with and without the speedups) in Section~\ref{sec: temporal complexity}. Section~\ref{sec: Numerical Stability} discusses its numerical properties and the devastating example from Noferini and Townsend \cite{Noferini}. Section~\ref{sec: Numerical Experiments} demonstrates the numerical properties of the algorithm, including comparisons between the methods and numerical exploration of the devastating example. We finish by summarizing directions for future work.

All the methods described here are implemented in Python 3 and are freely available at \url{https://github.com/tylerjarvis/eigen_rootfinding}.

\section{Background}

\subsection{Eigenvalue Methods for Rootfinding}
\label{sec: Moller--Stetter Rootfinding}

The companion matrix of a univariate polynomial $p\in \C[x]$ is a special matrix $C$ whose characteristic polynomial is $p$; and thus, the roots of $p$ are the eigenvalues of $C$, which can easily be computed numerically. The companion matrix also represents the linear operator of multiplication-by-$x$ on the finite-dimensional quotient algebra $\C[x]/(p)$.
This generalizes nicely to higher dimensions, in a construction due to M\"oller and Stetter \cite{Stetter,StetterBook,Moller}, which we now review briefly.

For a system of polynomials $p_1,\dots, p_n \in \C[x_1,\dots, x_n]$ satisfying the assumptions in Section~\ref{sec:assumptions}, consider the quotient algebra $\quotientalgebra = \C[x_1,\dots, x_n]/\I$, where $\I = (p_1,\dots, p_n)$ is the ideal generated by the polynomials. Under our assumptions the dimension of $\quotientalgebra$ as a vector space, is exactly equal to the number $r$ of common roots in $\C^n$ of the system \cite{Stetter}.  By B\'{e}zout's theorem, $r$ is no greater than  $\prod_{i=1}^n \deg(p_i)$, and for a generic family of polynomials, equality holds \cite[p.430]{UAG}: 
\[
r = \prod_{i=1}^n \deg p_i.
\]

For any $g\in \C[x_1,\dots, x_n]$, multiplication by $g$ defines a linear operator $m_g: \quotientalgebra \rightarrow \quotientalgebra$ that maps each $p\in \quotientalgebra$ to $pg$.  
Given a vector-space basis $\mathcal{B} = \{b_1, b_2, \dots, b_\numroots\}$ of $\quotientalgebra$, the operator $m_g$ has a matrix representation $M_g$, which we call the \emph{M\"{o}ller--Stetter matrix} of $g$.
It can be shown that if $\z$ is a common root of $p_1,\dots,p_n$, then $g(\z)$ is an eigenvalue of 
$M_g$.
If the values of $g$ at all of the $r$ roots are distinct, then $M_g$ is simple, and the row vector
\begin{equation}
\label{eq:left eigenvectors of MS matrix}
\begin{bmatrix}
b_1(\z) & b_2(\z) & \dots  & b_\numroots(\z)\end{bmatrix}
\end{equation}
is a left eigenvector associated with the eigenvalue $g(\z)$ (see  \cite[Chapter 2]{StetterBook}, \cite{Stetter}, or \cite[Chapter 4]{Cox2}). For a univariate polynomial of degree $d$ with the monomial basis $\mathcal{B}=\{1,x,x^2,\dots, x^{d-1}\}$, the matrix $M_x$ is the companion matrix. 

For solving a multivariate system, the M\"oller--Stetter matrices $M_{x_i}$ for $1\leq i \leq n$ are commonly used. The eigenvalues of $M_{x_i}$ are the $i$th coordinates of the roots, but they may not occur in the same order for each coordinate.
However, the matrices $M_{x_1}, \dots, M_{x_n}$ commute, so one method to find the zeros is to simultaneously diagonalize these $n$ commuting matrices to compute all $n$ coordinates of the roots \cite[p.4-5]{Telen}.

\subsection{Constructing a Basis for \texorpdfstring{$\quotientalgebra$}{A}}
The key to eigenvalue-based rootfinding is to construct an appropriate basis for the quotient algebra $\quotientalgebra$.  Gr\"obner and border bases are common choices when using exact arithmetic but are unstable when used with finite-precision arithmetic. However, there are some methods for trying to stably compute Gr\"obner and border bases in floating point arithmetic. These often use a combination of numerical and symbolic computations \cite{Kreuzer, Mourrain, Sasaki}.  

Telen and Van Barel devised a different method for constructing a basis of $\quotientalgebra$ \cite{Telen}, which we call \emph{direct Macaulay reduction}, described in Section~\ref{sec: Direct Macaulay Reduction}. Later, Telen, Mourrain, and Van Barel proposed a variant \cite{Telen4} that  we call \emph{null space Macaulay reduction} or simply the \emph{null space method}, described in Section~\ref{sec: Null Space Macaulay Reduciton}. The reasons for these names will become clear below.
Their methods construct the matrices $M_{x_i}$ in a way that is more stable than the methods using Gr\"obner or border bases \cite[p.16]{Telen}. Before describing both of these methods and two potential speedups (see Section~\ref{sec: Speedups}), we need to describe a fundamental tool they all have in common, namely, the Macaulay matrix. 

\subsubsection{The Macaulay Matrix}
\label{macaulay section}
A key tool in the methods used to construct a basis for $\quotientalgebra$ is the {Macaulay matrix}, which is constructed in a manner similar to the Sylvester matrix.
Given $p_1, p_2,\ldots, p_n \in \C[x_1,x_2,\ldots,x_n]$ and a positive integer $\dMac$,  the Macaulay matrix, $\Mac(\dMac)$, of degree $\dMac$ is constructed as follows.  The columns correspond to the various monomials in $\C[x_1,\dots, x_n]$ of total degree at most $\dMac$. The rows are coefficient vectors of polynomials of degree at most $\dMac$ of the form 
\begin{equation*}\label{eq:macaulay-rows}
x_1^{k_1} x_2^{k_2}\cdots x_n^{k_n} p_i
\end{equation*}
for some $i$ and some choice of positive integers $k_1,\dots, k_n$. The ordering of the rows is not important, but every such polynomial of degree at most $\dMac$ is represented in the Macaulay matrix.  

For example, given the system of polynomials
\[   \left\{
\begin{array}{ll}
      p_1 =& y^2 + 3xy - 4x + 1 \\
      p_2 =& -6xy - 2x^2 + 6y +3 \\
\end{array}
\right. 
\]
in $\C[x,y]$,
the degree-$3$ Macaulay matrix $\Mac(3)$ is as follows.
\[
\begin{blockarray}{ccccccccccc}
y^3 & xy^2  & x^2y  & x^3   & y^2   & xy    & x^2   & y     & x     & 1 \\[1.2ex]
\begin{block}{[cccccccccc]c}
    & -6    & -2    &       & 6     &       &       & 3     &       &   & yp_2\\
    &       & -6    & -2    &       & 6     &       &       & 3     &   & xp_2\\
    &       &       &       &       & -6    & -2    & 6     &       & 3 & p_2\\
1   &  3    &       &       &       & -4    &       & 1     &       &   & yp_1\\
    &  1    & 3     &       &       &       & -4    &       & 1     &   & xp_1\\
    &       &       &       & 1     & 3     &       &       & -4    & 1 & p_1\\
\end{block}
\end{blockarray}
\]
Note that every row corresponds to a polynomial in the ideal $\I$.  The Macaulay matrix is valuable because performing row operations on the Macaulay matrix produces new rows that still represent elements of $\I$. If the degree $d$ is large enough, then this can be used to identify polynomials that form a basis for $\quotientalgebra$. For more details on the construction of the Macaulay matrix, see  \cite[p.9]{Macaulay}.

\section{Reduction Methods}
In this section we describe various methods for constructing a basis of $\quotientalgebra$ and for computing the M\"oller--Stetter matrices from the resulting basis.

\subsection{Direct Macaulay Reduction}
\label{sec: Direct Macaulay Reduction}
We now describe several methods for finding a basis $\quotientbasis$ for $\quotientalgebra$ directly from the Macaulay matrix $\Mac(\dMac)$ for $\dMac$ sufficiently large. We call these \emph{direct Macaulay reduction} methods.  These methods are variants on the method from \cite[p.9--12]{Telen}, where it is also shown that it suffices to take  
\begin{equation}\label{eq:dmac}
\dMac = 1 - n + \sum_{i=1}^n \deg p_i.
\end{equation}
In this case $\nullitymac{\dMac} = \numroots$.  From now on we always take $\dMac$ as given in Equation~\eqref{eq:dmac}. 

First, partition $\Mac(\dMac)$ into two submatrices $\Mac_1$ and $\Mac_2$, where $\Mac_1$ consists of the columns representing the degree-$d$ monomials and $\Mac_2$ corresponds to the rest of the columns (all lower-degree monomials). Perform a QR factorization to get $\Mac_1 = QR$. There are matrices $Z$ and $\Mac_3$ such that
\[Q\hermitian \begin{bmatrix}\Mac_1 & \Mac_2\end{bmatrix} =  
\begin{bmatrix}\hat{R}   & Z \\
                0           & \Mac_3
\end{bmatrix},
\]
where $\hat{R}$ is the invertible submatrix of $R$. The assumptions in Section~\ref{sec:assumptions} guarantee that $R$ is full rank.

Now factor $\Mac_3$ to get $\Mac_3 = X V\hermitian$, where $X$ is easy to convert to RREF and $V$ is some unitary transition matrix which maps the standard monomial basis for $\lowdegring$ (the polynomials of degree at most $d-1$) to a new basis $\ringbasis$. In \cite{Telen}, $\Mac_3$ is factored using QR with pivoting. We suggest using an SVD factorization instead for reasons discussed below. One could also use an LQ factorization.  In any case, it follows that 
\begin{equation}
\begin{bmatrix}\hat{R}   & Z \\
                0           & \Mac_3
\end{bmatrix}
\begin{bmatrix} I & 0 \\
                0 & V
\end{bmatrix}   
= \begin{bmatrix}\hat{R}   & ZV \\
                0          & X
\end{bmatrix}.
\end{equation}
Multiplying the rightmost columns of $\Mac(\dMac)$ by $V$ means that those columns now represent polynomials in the basis $\ringbasis$ instead of the original monomials.

After reducing $X$ to echelon form, removing rows of zeros at the bottom of the matrix, and performing back-substitution, we get what we call a \emph{reduced Macaulay matrix}.
The polynomials in $\ringbasis$ corresponding to the free columns of this reduced matrix form a basis $\quotientbasis$ for $\quotientalgebra$.

Different factorizations of $\Mac_3$ lead to different bases $\quotientbasis$. The pivoted QR factorization gives $V\hermitian = P\transpose$, so $\quotientbasis$ will be a monomial basis. Using an SVD gives $XV\hermitian=U\Sigma V\hermitian$, which allows some simplifications in reducing $X$. In particular, we have
\[
\begin{bmatrix} I   & 0 \\
                0   & U\hermitian
\end{bmatrix}
\begin{bmatrix}\hat{R}   & ZV \\
                0          & X
\end{bmatrix}
=\begin{bmatrix}\hat{R}   & Z_1 & Z_2 \\
                0          & \hat\Sigma & 0\\
                0 & 0 & 0
\end{bmatrix}
\]
where $\hat\Sigma$ is the nonzero diagonal submatrix of $\Sigma$. Of course there could be some difficulty in numerically determining the rank of $\Sigma$. However, we have assumed there are exactly $r$ roots of the system, so $\Mac_3$ has nullity $r$.
Since rows in the resulting matrix are in $\I$ and $\hat{\Sigma}$ is diagonal, every basis element corresponding to a column of $\hat{\Sigma}$ is in $\I$,
and all the relevant information from the Macaulay matrix can be obtained by backsolving the top portion of the matrix to get
\[
\begin{bmatrix} I & 0 & \hat{R}^{-1}Z_2 
\end{bmatrix}.
\]
The LQ factorization has similar properties to the SVD. In practice, using the SVD gives the most accurate results of the three potential factorizations without sacrificing speed. See Section~\ref{sec: Numerical Experiments} and Figure~\ref{fig:QRPvsSVD_res/eigs}.

Regardless of the factorization, the polynomials in $\quotientbasis$ all have degree strictly less than $\dMac$, so for $i=1,\ldots,n$, the polynomial $x_ib$ is of degree at most $\dMac$. Therefore, we can express $x_ib$ in terms of $\quotientbasis$ using the transition matrix and relations from the reduced Macaulay matrix. This makes it possible to construct M\"oller--Stetter matrices $M_{x_i}$.

For example, if we compute $\quotientbasis$ using the SVD factorization, we form the matrix 
$$F = \begin{bmatrix} 
        -\hat{R}^{-1}Z_2 \\              
        V_{:,-r:}
      \end{bmatrix}$$
whose rows show how to represent each monomial of degree at most $\dMac$ in terms of $\quotientbasis$. To compute $M_{x_i}$, we must determine how to express $x_i\mu$ in terms of $\quotientbasis$ for each monomial $\mu$ of degree strictly less than $\dMac$. Multiplication by $x_i$ can be performed symbolically by extracting the rows of $F$ corresponding to $x_i\mu$. Let $\texttt{idx}_i$ be the indices of these rows. Then since multiplication by $(V_{:,-\numroots})\hermitian$ maps from the standard basis for $\lowdegring$ to $\quotientbasis$, it follows that
\[
M_{x_i} = (V_{:,-\numroots})\hermitian F_{\texttt{idx}_i:,:}.
\]
This is detailed in Algorithm~\ref{alg: MS Matrix}.

The eigenvalues of $M_{x_i}$ are the $i$th coordinates of the roots of the system. To extract the coordinates in their corresponding ordered tuples, one should diagonalize the $M_{x_i}$ matrices simultaneously. Unfortunately, many human-generated problems have eigenvalues with multiplicity greater than one in one or more coordinates, and thus are not uniquely diagonalizable. In order to avoid this, first perform a random orthogonal\footnote{One could use a more general unitary matrix here, but because many of the systems we want to solve are real, with real roots, it is a little cleaner to use a real orthogonal matrix instead.} 
change of coordinates $W$, to obtain new (rotated) M\"{o}ller--Stetter matrices $M_{y_j} = \sum_{i=1}^{n} W_{ji}M_{x_i}$
expressed in terms of new coordinates, $y_1,\ldots,y_n$. 

To perform the imultaneous diagonalization, Telen and Van Barel use a canonical polyadic decomposition (CPD), also known as CANDECOMP or PARAFAC, of the tensor formed by stacking an $\numroots\times \numroots$ identity matrix with $M_{y_1},\ldots, M_{y_n}$; see \cite[p.14]{Telen}. For more on the equivalence of CPD and simultaneous diagonalization see \cite{Lathauwer} or \cite[p.366]{BCS}. The standard implementations of CPD in Python performed poorly for us (they were both slow and inaccurate), so instead we use a Schur Decomposition $M_{y_1} = UTU\hermitian$. Because the matrices commute, $U\hermitian M_{y_j} U$ triangularizes $M_{y_j}$ for $j=1,\ldots,n$. The $k$th diagonal entry of $U\hermitian M_{y_j} U$ is the $y_j$-coordinates for the $k$th root of the system; in other words, $U\hermitian M_{y_j} U$ not only triangularizes the system, but also does so in such a way that preserves the ordering of the roots. 
While this triangularization is exact in theory, in practice more computational precision is gained by computing the eigenvalues of every $M_{y_j}$ independently (using QR iteration, for example), and then matching them to their nearest neighbor in the ordering given by the Schur Decomposition. Finally, the $y_j$-coordinates are rotated back to $x_i$-coordinates via left multiplication by $W\transpose$. This is detailed in  Algorithm~\ref{alg: Simul Diag}.

\begin{algorithm}[H]
\caption{Direct Macaulay Solver using SVD}
\label{alg: MS Matrix}
\begin{algorithmic}[1]
\Procedure{\texttt{MacaulaySVD}}{$p_1,...,p_n$}
\State $\dMac \leftarrow  1 - n + \sum_{i=1}^n \texttt{deg} (p_i)$ \Comment{degree of Macaulay matrix}
\State $\numroots \leftarrow  \prod_{i=1}^n \texttt{deg} (p_i)$ \Comment{number of roots}

\State $\Mac,\texttt{col} \leftarrow \texttt{macaulay}(p_1,...,p_n)$
\Comment{Macaulay matrix and column labels}
\State $\texttt{cut} \leftarrow {\binom{d+n-1}{d}}$ 
\Comment number of degree-$d$ columns
\State $\Mac_1 \leftarrow \Mac_{:,:\texttt{cut}}$
\Comment{split into high and low degree columns}
\State $\Mac_2 \leftarrow \Mac_{:,\texttt{cut}:}$
\State $Q,R \leftarrow  \texttt{qr}(\Mac_{1})$ 
\Comment{$QR$-factor}
\State $\hat{R} \leftarrow R_{:\texttt{cut},:}$
\Comment {$\hat{R}$ is nonzero rows of $R$}
\State $Z \leftarrow (Q\hermitian \Mac_2)_{:\texttt{cut},:}$
\Comment{desired part of $Q\hermitian \Mac_2 = 
 \begin{bmatrix} Z \\ \Mac_3
 \end{bmatrix} $}
\State $\Mac_3 \leftarrow (Q\hermitian \Mac_2)_{\texttt{cut}:,:} $
\vspace{.15cm}
\State $U, \Sigma, V^{H} \leftarrow
\texttt{svd}(\Mac_{3})$
\State $Z_2 \leftarrow ZV_{:,-\numroots:}$
\Comment{desired part of $ZV$ from $ 
    \begin{bmatrix} \hat{R} & ZV \\ 0 & U \Sigma 
    \end{bmatrix}
$}
\vspace{.15cm}
\State $\tilde Z_2 \leftarrow \hat R^{-1} Z_2$
\Comment{back substitution}
\State $F \leftarrow \begin{bmatrix} 
        -\hat{R}^{-1}Z_2 \\              
        V_{:,-r:}
      \end{bmatrix}$ \Comment{matrix to convert monomials to $\quotientbasis$}
\For{$i$ in $1, ..., n$}
\Comment{compute $M_{x_i}$}
\State $\texttt{idx}_i \leftarrow \texttt{get\_product\_idx}(i,\texttt{col},\texttt{cut})$
\Comment{shift column labels to multiply}
\State $M_{x_i} \leftarrow (V_{:,-\numroots})\hermitian F_{\texttt{idx}_i:,:}$
\EndFor
\State $\texttt{roots} \leftarrow \texttt{sim\_diag}(M_{x_1},\ldots,M_{x_n})$
\Comment{simultaneous diagonalization}\\ \Return \texttt{roots}
\EndProcedure
\end{algorithmic}
\end{algorithm}

\begin{algorithm}
\begin{algorithmic}[1]
\caption{Simultaneous Diagonalization Method}
\label{alg: Simul Diag}
\Procedure{\texttt{sim\_diag}}{$(M_{x_1},\ldots,M_{x_n})$}
\State $W \leftarrow \texttt{rand\_orthog\_matrix}(n)$ \Comment{choose a rotation}
\For{$j$ in $1,\ldots,n$}
\State $M_{y_j} \leftarrow \sum_{i=1}^{n} W_{ji}M_{x_i}$ \Comment{rotate coordinates}
\EndFor
\State $\texttt{roots} \leftarrow \texttt{empty}(n,\numroots)$
\Comment{initialize root array}
\State $U,T \leftarrow \texttt{schur}\left(M_{y_1}\right)$\Comment{Schur decomposition}
\State $\texttt{roots}_{1,:} \leftarrow \texttt{diag}(T)$ \Comment{$y_1$ coordinates of roots}
\For{$j$ in $2,\ldots,n$} \Comment{find remaining coordinates}
\State $\texttt{ordered\_eigs} \leftarrow \texttt{diag}\left(U\hermitian M_{y_j} U\right)$
\Comment{ordered to match $\texttt{roots}_{1,:}$}
\State $\texttt{unordered\_eigs} \leftarrow \texttt{eigvals} \left(M_{y_j}\right)$
\Comment{more precisely computed}
\State $\texttt{roots}_{j,:} \leftarrow \texttt{match\_eigs}(\texttt{ordered\_eigs},\texttt{unordered\_eigs})$
\EndFor
\State \Return $W\transpose \texttt{roots}$
\Comment{rotate coordinates back}
\EndProcedure
\end{algorithmic}
\end{algorithm}

\subsection{Null Space Macaulay Reduction}
\label{sec: Null Space Macaulay Reduciton}
In \cite{Telen4}, Mourrain, Telen, and Van Barel demonstrated a different way to construct M\"{o}ller--Stetter matrices using the Macaulay matrix. By restricting the null space of the Macaulay matrix to certain known subspaces, we can directly create the M\"oller--Stetter matrices representing multiplication by the monomials $x_{1}, \dots x_{n}$. One advantage of using this \emph{null space method} over the direct Macaulay reduction outlined earlier is a potential speed increase as described in Section~\ref{sec: Degree by Degree}. For a performance comparison of the various methods see Section~\ref{sec: Numerical Experiments}. 
    
The first step in the null space method is to construct a matrix $N$ whose columns form a basis for the null space of the Macaulay matrix $\Mac(\dMac)$, with $\dMac = 1 - n + \sum_{i=1}^n \deg p_i$ as before. 
Split $N\hermitian$ into submatrices $N_1$ and $N_2$ where $N_1$ contains the columns corresponding to degree-$d$ monomials, and $N_2$ contains the rest of the columns. To find a basis $\quotientbasis$ for $\quotientalgebra$, compute a factorization $N_2 = XV\hermitian$ where, similar to direct Macaulay reduction, $X$ is easy to convert to RREF and $V$ is unitary. However, $\quotientbasis$ now corresponds to the pivot columns in $N\hermitian$ instead of free columns. Using this factorization, one can construct a matrix $F$ that converts each monomial to its representation in $\quotientbasis$ in order to build M\"{o}ller--Stetter matrices.

For example, using an SVD factorization, we take $N_2 = U \Sigma V\hermitian$ and denote the nonzero invertible submatrix of $\Sigma$ by $\hat\Sigma$. Then $\Sigma V\hermitian = \hat \Sigma (V_{:,:r})\hermitian$ and 
$$F\hermitian = \hat\Sigma^{-1} U\hermitian N = \begin{bmatrix}
    \hat\Sigma^{-1} U\hermitian N_1 & (V_{:,:r})\hermitian
    \end{bmatrix}.$$
Of course, a compact SVD factorization would suffice for this computation.
A similar process can be used to compute $\mathcal B$ and $F$ from an LQ or QRP factorization of $N_2$.
 For more details on using null spaces of $\Mac(\dMac)$ to compute M\"{o}ller--Stetter matrices, see \cite{Telen2,Telen4}. Null space computations can be expensive and slow. In the following two subsections we briefly review the two speedups given in \cite{Telen4}.
  
\section{Speedups}
\label{sec: Speedups}

\subsection{Degree by Degree construction}
\label{sec: Degree by Degree}
One way that Telen, Mourain, and Van Barel propose to compute $N$ more efficiently is to exploit the fact that certain submatrices of a Macaulay matrix are lower-degree Macaulay matrices. The \emph{degree by degree} method iteratively constructs $\Mac(d_{k+1})$ and its null space $N_{k+1}$ from a lower-degree Macaulay matrix $\Mac(d_k)$ and its null space $N_{k}$. The sequence of degrees $d_k$ could use any increment, but we choose $d_{k+1} = 1+d_k$ in our numerical experiments and complexity analysis. This iterative process continues until the Macaulay matrix $\Mac(\dMac)$ and its null space $N_{\dMac}$ are computed, possibly saving time from computing $\Mac(\dMac)$ and $N_{\dMac}$ directly. In this section we give the details of this approach.

To build $\Mac(d_{k+1})$ from $\Mac(d_k)$, observe that $\Mac(d_k)$ is a submatrix of $\Mac(d_{k+1})$. As an example, consider the system of two equations used earlier.
\[   \left\{
\begin{array}{ll}
      p_1 =& y^2 + 3xy - 4x + 1 \\
      p_2 =& -6xy - 2x^2 + 6y +3 \\
\end{array}
\right.
\]
The degree-2 Macaulay matrix $\Mac(2)$ is as follows.
\[
\begin{blockarray}{ccccccc}
y^2   & xy    & x^2   & y     & x     & 1 \\
\begin{block}{[cccccc]c}
 1&      3&       &       &     -4&      1&   p_1\\
   &     -6&     -2&      6&       &      3&   p_2\\
\end{block}
\end{blockarray}
\]
Now compare this with $\Mac(3)$, slightly reordering the rows from when this matrix was presented earlier. 
\[
\begin{blockarray}{ccccccccccc}
y^3 & xy^2  & x^2y  & x^3   & y^2   & xy    & x^2   & y     & x     & 1 \\
\begin{block}{[cccccccccc]c}
    &       &       &       & 1     & 3     &       &       & -4    & 1 & p_1\\
    &       &       &       &       & -6    & -2    & 6     &       & 3 & p_2\\
    & -6    & -2    &       & 6     &       &       & 3     &       &   & yp_2\\
    &       & -6    & -2    &       & 6     &       &       & 3     &   & xp_2\\
1   &  3    &       &       &       & -4    &       & 1     &       &   & yp_1\\
    &  1    & 3     &       &       &       & -4    &       & 1     &   & xp_1\\
\end{block}
\end{blockarray}
\]
Notice the first two rows of $\Mac(3)$ consist of $\Mac(2)$ with 0's in the higher degree columns not represented in $\Mac(2)$. Furthermore, the other rows of $\Mac(3)$ are just the entries of $\Mac(2)$ translated into appropriate column placements based on which monomial we multiply $p_i$ by. We let $B$ denote the portion of the rows beneath $\Mac(2)$, and we let $A$ denote the portion beneath the zero block. 
Thus we can construct $\Mac(3)$ from $\Mac(2)$ without actually performing polynomial-monomial multiplication. 
More generally, if we have the Macaulay matrix $\Mac(d_{k})$ for some degree $d_{k}$, we have 
$$\Mac(d_{k+1}) = \begin{bmatrix}
0 & \Mac(d_k) \\
A_k & B_k
\end{bmatrix}$$ 
where $A_k$ and $B_k$ can be easily obtained from $\Mac(d_{k})$.
 
We now move to the problem of computing $N_{k+1}$ from $N_k$. In a slight abuse of notation, let $N_k$ be a matrix representation of a basis for the null space of $\Mac(d_k)$. 
Define $$\hat{N}_{k+1} = \begin{bmatrix}
 I & 0 \\ 0 & N_k
 \end{bmatrix},$$
where the identity matrix has dimension equal to the number of monomial columns added when creating $\Mac(d_{k+1})$ from $\Mac(d_k)$.
Let $N_{k+1} = \hat{N}_{k+1} L_{k+1}$
where $L_{k+1}$ is a matrix whose columns span the kernel of
$$\hat{N}_{k+1}
 \begin{bmatrix}
 A_k \\ B_k
 \end{bmatrix}
 = \begin{bmatrix}
 A_k & B_k N_k
 \end{bmatrix}.$$
Then  
\begin{equation*}
\Mac(d_{k+1})N_{k+1}
=\begin{bmatrix}
0 & \Mac(d_k) \\
A_k & B_k
\end{bmatrix}\begin{bmatrix}
 I & 0 \\ 0 & N_k
 \end{bmatrix}
 L_{k+1} = 0
\end{equation*}
and so $N_{k+1}$ spans the null space of $\Mac(d_{k+1})$.
 
 To compute $L_{k+1}$ we must determine the nullity of $\begin{bmatrix}
 A_k & B_k N_k
 \end{bmatrix}$, which is not numerically straightforward, but can be computed using the following lemma.
 
\begin{proposition}\label{lem:koszul-nullity}
Let $\pd_i = \deg(p_i)$. Using the notation of this section, the nullity of $\begin{bmatrix}
 A_k & B_k N_k
 \end{bmatrix}$ is
 \begin{align}
 \nullity\left(\begin{bmatrix}
 A_k & B_k N_k
 \end{bmatrix}\right) & =  \nullitymac{d_{k+1}}\label{eq:nullity}\\
 &= \sum_{j=0}^n (-1)^j \sum_{i_1<\cdots<i_j} \binom{n+d_{k+1} - \sum_{\ell=1}^j \pd_{i_\ell}}{n},\notag
 \end{align}
 where  we set $\binom{a}{b} = 0$ unless $a,b\ge 0$ and $a\ge b$.  Moreover, \[
 \nullitymac{\dMac} = \nullitymac{\dMac-1} = r = \prod_{i=1}^n \pd_i,
 \]
 where $\dMac$ is given in Equation~\eqref{eq:dmac}.
\end{proposition} 
\begin{proof}
 First, note that  $\rank(L_{k+1}) \le \nullitymac{d_{k+1}}$ because $\begin{bmatrix}
 I & 0 \\ 0 & N_k
 \end{bmatrix}
 L_{k+1}$ is in the kernel of $\Mac(d_{k+1})$.  
 Conversely, if $(\mathbf{x}_1, \mathbf{x}_2)$ is in the kernel of $\Mac(d_{k+1})$, then $\mathbf{x}_2$ is annihilated by $\Mac(d_k)$, and hence must be of the form $N_k\mathbf{y}$ for some $\mathbf{y}$.  This shows that $(\mathbf{x}_1, \mathbf{y})$ is in the range of $L_{k+1}$, hence $\rank(L_{k+1}) = \nullitymac{d_{k+1}}$. 
 
 The nullity of $\Mac(d_k)$ can be computed from the Hilbert function and the Koszul complex.  Let $\pring = \C[x_1,\dots, x_n]$, considered as a graded $\C$-algebra.  For any graded $\pring$-module $A$ and any $t\in \N$ let $A_{\le t}$ denote the subspace of all elements of degree at most $t$.
 For any $k\in \Z$ let $A(a)$ be $A$ with its grading shifted by $a$.  The Koszul complex of $p_1,\dots, p_n$ is the graded complex
 \begin{align*} 
 \cdots & \rTo^{f_4} \bigoplus_{i_1<i_2< i_3} \pring\left(-\sum_{\ell=1}^3\pd_{i_\ell}\right) \rTo^{f_3}  \bigoplus_{i_1<i_2} \pring(-\pd_{i_1} - \pd_{i_2})\\ & \rTo^{f_2}
 \bigoplus_{i=1}^n \pring(-\pd_i) \rTo^{f_1}  \pring \rTo \quotientalgebra \rTo 0,
 \end{align*}
 where $f_1$ maps any $q \in \pring\left(-\pd_{i}\right)$ to $q p_{i}\in \pring$, and the other $f_k$ are defined as an appropriate alternating sum of similar terms (see \cite[Chapter 17]{Eisenbud} or \cite[Chapter 6]{Cox2}). Specifically, the image of $f_1$ is the ideal $\I$, and the row space of $\Mac(d_k)$ 
 corresponds to the image of  $\bigoplus_{j=1}^n \pring(-\pd_j)_{\le d_k}$ under the map $f_1$.  Moreover the space \(\pring_{d_k}\) is spanned by the monomials that correspond to the columns of $\Mac(d_k)$.
 Thus  $\nullitymac{d_k}$ is the dimension of the subspace $\quotientalgebra_{\le d_k}$ of the quotient algebra $\quotientalgebra$ spanned by monomials of degree at most $d_k$.

Our assumptions on $\I=(p_1,\dots, p_n)$ guarantee that the graded Koszul complex for $p_1,\dots, p_n$ is exact, therefore \(\dim(\quotientalgebra_{\le d_k}) \) is the alternating sum of the corresponding dimensions of the terms in the Koszul complex.  It is straightforward to verify that \(\dim(\pring(-a)_{\le t}) = \binom{n-a + t}{n}\) for all $n,a,t\in \N$, from which the Equation~\ref{eq:nullity} follows.

Finally, the Hilbert polynomial $P_{\quotientalgebra} (t)$ of $\quotientalgebra$ is constant $P_{\quotientalgebra} (t) = r = \prod_{i=1}^n \beta_i$, and the Hilbert function $\varphi_{\quotientalgebra}(t)$ of $\quotientalgebra$ agrees with $P_{\quotientalgebra} (t)$ whenever $t$ is large enough that the terms $\binom{n+ t - \sum_{\ell=1}^j \pd_{i_\ell}}{n}$ can all be written as polynomials
\[
\binom{n + t - \sum_{\ell=1}^j \pd_{i_\ell}}{n} = \frac{\prod_{j=0}^{n-1} (n + t - j - \sum_{\ell=1}^j \pd_{i_\ell} )}{n!}.
\]
It's enough to check this condition for the final term 
\begin{equation}\label{eq:binom-polynom}
\binom{n + t - \sum_{i=1}^n \pd_{i}}{n} = \frac{\prod_{j=0}^{n-1} (n + t - j - \sum_{i=1}^n \pd_{i} )}{n!}.
\end{equation}
This clearly holds whenever $\sum_{i=1}^n \pd_{i}\le t$, by the definition of the binomial coefficient.  But it also holds when $    \sum_{i=1}^n \pd_{i} - n \le t < \sum_{i=1}^n \pd_{i}$  because both the binomial coefficient and the polynomial on the right side of Equation~\eqref{eq:binom-polynom} vanish for these values of $t$.  Thus \[
\nullitymac{\dMac-1} = \dim(\quotientalgebra_{\le \dMac-1}) = \varphi_{\quotientalgebra}(d-1) = P_{\quotientalgebra}(d-1) = r = \nullitymac{\dMac}, 
\]
as required.
\end{proof}

 The main advantage of the degree by degree construction is avoiding the costly computation of the null space of the entire matrix $\Mac(\dMac)$, usually done by computing the SVD, and instead performing many smaller calculations. This makes it a potential improvement for the null space Macaulay methods but not the direct Macaulay reduction methods. For more details about this method, see \cite{Telen4}.
 
\subsection{Random Combinations}
\label{sec: Random Combinations}
One can take advantage of the structure of the Macaulay matrix to reduce its size. The Macaulay matrix is row rank deficient. Because every row of the Macaulay matrix represents a polynomial in the ideal, any linear combination of the rows also represents a polynomial in the ideal. Thus we can take $\rank\left(\Mac(\dMac)\right)$ random linear combinations of the rows of $\Mac(\dMac)$, and get a matrix with the same rank and kernel as $\Mac(\dMac)$. 
One way to do this is to let $C$ be a  
\[
\left[\binom{\dMac + n}{n} - r \right]\times \sum_{i=1}^{n}\binom{d - \deg p_{i} + n}{n}
\]
matrix with entries drawn from the standard normal distribution.
With probability one, the product matrix $C \Mac(\dMac)$ has full row rank and has the same nullspace as $\Mac(\dMac)$.

This new matrix is smaller than $\Mac(\dMac)$, and it preserves the range and the kernel. Direct Macaulay reduction and null space methods can then be applied to this new matrix. In our numerical experiments, we found that random combinations improved the speed of the direct Macaulay methods more than it improved the null space methods. However, this smaller matrix may or may not behave well in calculations; see Section~\ref{sec: compare methods}. 

\section{Temporal Complexity}
\label{sec: temporal complexity}

\par In this section, we compute the temporal complexities of the various algorithms discussed in this paper. For this section, we assume that the factorization step in the direct Macaulay reduction method and the null space Macaulay reduction method uses the SVD variant. We do this in part because a singular value decomposition has the same asymptotic (big-O) complexity as an LQ or QRP factorization, 
 but also because our numerical experiments found that using the SVD gives the best results without sacrificing speed. See Section~\ref{sec: SVD vs QRP}.

\subsection{Background and Assumptions}
We only show the complexities of the algorithms up to the reduction step (i.e., forming $F$), but not forming M\"{o}ller-Stetter matrices or finding the roots. This is because once the reduction step is complete, each method constructs the M\"{o}ller-Stetter matrices and extracts the roots in the same way. By comparing with the complexities presented below, it is easy to verify that forming M\"{o}ller-Stetter matrices and computing eigenvalues is asymptotically less expensive than the reduction step.

Because of the nature of the rootfinding problem, there are actually two variables to consider when computing the temporal complexity: degree and dimension. Rather than letting both degree and dimension go to infinity simultaneously, we give two asymptotic bounds per algorithm; one bound is for fixed dimension, increasing degree, and the other is for increasing dimension, fixed degree.
For simplicity, we assume we are given a system of polynomials of the same degree $\degree$ in $\dimension$ dimensions with $\degree > 1$ and $\dimension > 1$.

We define a \emph{tight asymptotic bound} for
$f(\dimension,\degree)$ as $\degree\rightarrow\infty$ to be a function $g(\dimension,\degree)$ such that $f = O(g)$ and $f = \Omega(g)$. Intuitively, this is a bound that cannot be improved. Formally, we have 
$$0<\degreeliminf \frac{f(\degree,\dimension)}{g(\degree,\dimension)}\leq \degreelimsup \frac{f(\degree,\dimension)}{g(\degree,\dimension)} < \infty.$$
If $\degreelim \frac{f(\degree,\dimension)}{g(\degree,\dimension)}$ exists, this is equivalent to $f \sim Cg$ for some constant $C>0$.
We can similarly define tight asymptotic bounds as $\dimension\rightarrow\infty$.

Finally, we use the convention that $\binom{a}{b} = 0$ if $a < 0$, $b < 0$ or $a < b$, i.e. if it is not well defined.

\subsection{Basic Asymptotic Complexities}
We briefly summarize the complexity of the major linear algebra routines within our algorithm. 
\begin{itemize}
    \item The complexity of computing the QR factorization of an $m \times n$ matrix is $O(mn^2)$. We denote this $\QR(m,n) = mn^2$.
\item The complexity of computing the SVD of an $m \times n$ matrix is $O(mn^2)$, assuming $m \geq n$, so in general it is $mn\min(m,n)$. We denote this $\SVD(m,n) = mn\min(m,n)$.

\item The complexity of matrix multiplication of a dense $m \times n$ and a dense $n \times k$ matrix is $O(mnk)$. We denote this $\MatMult(m,n,k) = mnk$. We recognize that matrix multiplication can be done with sub-cubic complexity, but most implementations use the simple cubic method.

\item The complexity of backsubstitution on a triangular $n \times n$ matrix against a $n \times m$ matrix is $O(mn^2)$. We denote this $\Backsolve(n,m) = mn^2$.
\end{itemize}

\subsection{Variable Definitions}

Let $\dMac$ be the Macaulay degree, $\dMac = \dimension\degree - \dimension + 1$ as mentioned in Section~\ref{sec: Direct Macaulay Reduction}.
We use the notation notation from \cite{Telen4}, with the following variables:

\begin{itemize}
    \item $\monomialsdeg{k}$ is the number of monomials of degree $k$, which is equal to $\binom{\dimension+k-1}{k}$.

    \item $\monomialsleqdeg{k}$ is the number of monomials of degree less than or equal to $k$, which is equal to $\binom{\dimension+k}{k}$.

    \item $\polysdeg{k}$ is the number of polynomials (rows) of degree $k$ in a Macaulay matrix of degree at least $k$. It is equal to $\sum_{i=1}^{\dimension}\monomialsleqdeg{k-\degree_i} = \dimension \monomialsleqdeg{k-\degree}$.  We have that $\polysdeg{k} = 0$ for $k < \degree$.

    \item $\polysinmac{k}$ is the number of rows in $\Mac(k)$. It is equal to $\sum_{i=1}^{k}\polysdeg{i} = n\sum_{i=\degree}^{k}\monomialsleqdeg{i-\degree}$.

\end{itemize}

Finally, we define a variable, $\alpha_k$ that will be used in examining the complexity of systems with constant degree and varying dimension.
\begin{definition}\label{def:alpha}
For $k \geq 2$, let $\alpha_k = (\frac{k}{k-1})^{k-1}$. Note that $\alpha_2 = 2$ and $\alpha_k$ is an increasing sequence with limit $e$.
\end{definition}

\subsection{Variable Asymptotic Bounds}
\label{sec: temporal complexity: variable bounds}

Tight asymptotic bounds for several relevant variables are summarized in Table~\ref{table: complexity bounds}. For proofs see Appendix~\ref{appendix: complexity proofs}. We combine these bounds into bounds on the complexity of each algorithm in sections~\ref{sec: temporal complexity: Direct Macaulay}-\ref{sec: temporal complexity: degree-by-degree}.

\begin{table}
    \centering
\begin{tabular}{ c|c|c } 
 Term 
 & fixed $\dimension$ and $\degree\rightarrow\infty$ 
 & fixed $\degree$ and $\dimension\rightarrow\infty$ \\ 
 \hline
 $\monomialsleqdeg{\dMac-1}$ & $\degree^\dimension$ & $\frac{1}{\sqrt{n}} \degree ^ \dimension \alpha_{\degree} ^ \dimension$ \\ 
 $\monomialsleqdeg{\dMac}$ & $\degree^\dimension$ & $\frac{1}{\sqrt{n}} \degree ^ \dimension \alpha_{\degree} ^ \dimension$ \\ 
 $\monomialsdeg{\dMac}$ & $\degree^{\dimension-1}$ & $\frac{1}{\sqrt{n}} \degree ^ \dimension \alpha_{\degree} ^ \dimension$ \\ 
  $\polysdeg{\dMac}$  & $\degree^\dimension$ & $\sqrt{n} \degree ^ \dimension \alpha_{\degree} ^ \dimension$ \\
 $\polysinmac{\dMac}$  & $\degree^{\dimension+1}$ & $\sqrt{n} \degree ^ \dimension \alpha_{\degree} ^ \dimension$ \\
 $\numroots$  & $\degree^{\dimension}$ & $\degree^{\dimension}$
\end{tabular}

\caption{Table summarizing the tight asymptotic bounds for each term relevant in the complexity analysis. Here, as defined in Definition~\ref{def:alpha}, we use $\alpha_\beta = \left(\frac{\beta}{\beta-1}\right)^{\beta-1}$, so $2 \le \alpha_\beta \le e$ for all $\beta > 1$. Of course, under our simplifying assumptions of this section $\numroots=\degree^\dimension$, but it is listed here for convenience. \label{table: complexity bounds}}
\end{table}

\subsection{Direct Macaulay SVD}
\label{sec: temporal complexity: Direct Macaulay}
The main steps of the direct Macaulay SVD method (see Section~\ref{sec: Direct Macaulay Reduction}) are as follows:
\begin{enumerate}
  \item Compute a QR decomposition of $\Mac_1$. This is $QR(\polysinmac{\dMac}, \monomialsdeg{\dMac})$.
  \item Multiply $Q\hermitian \Mac_2$. This is $\MatMult(\polysinmac{\dMac}, \polysinmac{\dMac}, \monomialsleqdeg{\dMac-1})$.
  \item Compute an SVD of $\Mac_3$. This is $\SVD(\polysinmac{\dMac}-\monomialsdeg{\dMac}, \monomialsleqdeg{\dMac-1})$.
  \item Multiply $ZV_{:,-\numroots:}$. This is $\MatMult(\monomialsdeg{\dMac}, \monomialsleqdeg{\dMac-1}, \numroots)$.
  \item Backsolve $R^{-1} Z_2$. This is $\Backsolve(\monomialsdeg{\dMac}, \numroots)$.
\end{enumerate}

Summing these gives a complexity of 
\begin{align*}
&\polysinmac{\dMac} \monomialsdeg{\dMac}^2 \\
&+ 
\polysinmac{\dMac}^2\monomialsleqdeg{\dMac-1} \\
&+ 
(\polysinmac{\dMac}-\monomialsdeg{\dMac})\monomialsleqdeg{\dMac-1}\min(\polysinmac{\dMac}-\monomialsdeg{\dMac}, \monomialsleqdeg{\dMac-1}) \\
&+ 
\monomialsdeg{\dMac} \monomialsleqdeg{\dMac-1} \numroots \\
&+
\monomialsdeg{\dMac}^2 \numroots.
\end{align*}
Tight asymptotic bounds in dimension and degree can be found by combining bounds for each term.
\begin{itemize}
\item For fixed $\dimension$ and $\degree\rightarrow\infty$, it is straightforward to verify that this becomes
$O(\degree^{3\dimension+2}).$
\item For fixed $\degree$ and $\dimension\rightarrow\infty$, a straightforward computation shows that the complexity is
$O(\sqrt{n}\beta^{3n}\alpha_{\beta}^{3n}).$
\end{itemize}

\subsection{Null Space Macaulay SVD}
\label{sec: temporal complexity: null space}
The main steps of the null space Macaulay SVD (see Section~\ref{sec: Null Space Macaulay Reduciton}) are as follows:
\begin{enumerate}
  \item Perform an SVD on the Macaulay Matrix. This is $\SVD(\polysinmac{\dMac}, \monomialsleqdeg{\dMac})$.
  \item Perform an SVD on $N_2$. This is $\SVD(\numroots, \monomialsleqdeg{\dMac-1})$.
  \item Multiply $U\hermitian N_1$. The multiplication by $\hat\Sigma^{-1}$ is clearly of lower complexity and will not be counted. This is $\MatMult(\numroots, \numroots, \monomialsdeg{\dMac})$.
\end{enumerate}

Summing these gives a complexity of 
\begin{align*}
    &\polysinmac{\dMac} \monomialsleqdeg\dMac \min(\polysinmac{\dMac}, \monomialsleqdeg\dMac) \\ &+ 
    \numroots \monomialsleqdeg{\dMac-1} \min(\numroots, \monomialsleqdeg{\dMac-1}) \\ &+ 
    \numroots^2 \monomialsdeg{\dMac}.
\end{align*}
\begin{itemize}
\item For fixed $\dimension$ and $\degree\rightarrow\infty$, a tight asymptotic bound is
$O(\degree^{3\dimension+1}),$
which is cheaper than direct Macaulay SVD by a factor of $\degree$.
\item For fixed $\degree$ and $\dimension\rightarrow\infty$, a tight bound is
$O(\sqrt{n}\beta^{3n} \alpha_{\beta}^{3n}),$
which is the same as for the direct Macaulay reduction method from the previous section.
\end{itemize}

\subsection{Random Combinations}
\label{sec: temporal complexity: rand combos}
Both direct and null space random combinations methods starts with a matrix multiplication that reduces the size of the Macaulay matrix (see Section~\ref{sec: Random Combinations}). This is $\MatMult(\monomialsleqdeg\dMac - \numroots, \polysinmac{\dMac}, \monomialsleqdeg\dMac)$.
We then can do either the direct Macaulay SVD or null space SVD reduction with the number of rows being $\monomialsleqdeg\dMac - \numroots$ instead of $\polysinmac{\dMac}$, so we can just use our previous analysis but replace each $\polysinmac{\dMac}$ by $\monomialsleqdeg\dMac - \numroots$ and add on a first step of $\MatMult(\monomialsleqdeg\dMac - \numroots, \polysinmac{\dMac}, \monomialsleqdeg\dMac)$.
\subsubsection{Direct Macaulay SVD Random Combinations}
\begin{enumerate}
   \item Matrix multiplity to reduce the size of the Macaulay Matrix. This is $\MatMult(\monomialsleqdeg\dMac - \numroots, \polysinmac{\dMac}, \monomialsleqdeg\dMac)$.
   \item Perform a QR on $\Mac_1$. This is $QR(\monomialsleqdeg\dMac - \numroots, \monomialsdeg{\dMac})$.
  \item Multiply $Q\hermitian \Mac_2$. This is $\MatMult(\monomialsleqdeg\dMac - \numroots, \monomialsleqdeg\dMac - \numroots, \monomialsleqdeg{\dMac-1})$.
  \item Perform an SVD on $\Mac_3$. This is $\SVD(\monomialsleqdeg\dMac - \numroots-\monomialsdeg{\dMac}, \monomialsleqdeg{\dMac-1})$.
  \item Multiply $ZV_{:,-\numroots:}$. This is $\MatMult(\monomialsdeg{\dMac}, \monomialsleqdeg{\dMac-1}, \numroots)$.
  \item Backsolve $R^{-1} Z_2$. This is $\Backsolve(\monomialsdeg{\dMac}, \numroots)$.
\end{enumerate}
Summing these gives a complexity of
\begin{align*}
&\left(\monomialsleqdeg\dMac - \numroots\right) \polysinmac{\dMac}\monomialsleqdeg\dMac) \\
&+ (\monomialsleqdeg\dMac - \numroots) \monomialsdeg{\dMac}^2 \\
&+ (\monomialsleqdeg\dMac - \numroots)^2\monomialsleqdeg{\dMac-1} \\
&+ (\monomialsleqdeg\dMac - \numroots-\monomialsdeg{\dMac})\monomialsleqdeg{\dMac-1}\min(\monomialsleqdeg\dMac - \numroots-\monomialsdeg{\dMac}, \monomialsleqdeg{\dMac-1}) \\
&+ \monomialsdeg{\dMac} \monomialsleqdeg{\dMac-1} \numroots \\
&+ \monomialsdeg{\dMac}^2 \numroots.
\end{align*}
\begin{itemize}
\item For fixed $\dimension$ and $\degree\rightarrow\infty$, a tight asymptotic bound is
$O(\degree^{3\dimension+1}),$
which is cheaper than Direct Macaulay SVD by a factor of $\degree$.
\item For fixed $\degree$ and $\dimension\rightarrow\infty$, a tight bound is
$O(\dimension^{-1/2}\beta^{3\dimension}\alpha_{\beta}^{3\dimension}),$
which is cheaper by a factor of $\dimension$ than null space or direct Macaulay methods.
\end{itemize}
\subsubsection{Nullspace SVD Random Combinations}
\begin{enumerate}
  \item Matrix multiply to reduce the size of the Macaulay Matrix. This is $\MatMult(\monomialsleqdeg\dMac - \numroots, \polysinmac{\dMac}, \monomialsleqdeg\dMac)$.
  \item Perform an SVD on the smaller Macaulay Matrix. This is $\SVD(\monomialsleqdeg\dMac - \numroots, \monomialsleqdeg\dMac)$.
  \item Perform an SVD on $N_2$. This is $\SVD(\numroots, \monomialsleqdeg{\dMac-1})$.
  \item Multiply $U\hermitian N_1$. The multiplication by $\hat\Sigma^{-1}$ is clearly of lower complexity and will not be counted. This is $\MatMult(\numroots, \numroots, \monomialsdeg{\dMac})$.
\end{enumerate}
Summing these gives a complexity of
\begin{align*}
& (\monomialsleqdeg\dMac - \numroots) \polysinmac{\dMac} \monomialsleqdeg\dMac \\
&+ (\monomialsleqdeg\dMac - \numroots) \monomialsleqdeg\dMac \min((\monomialsleqdeg\dMac - \numroots), \monomialsleqdeg\dMac) \\
&+ \numroots \monomialsleqdeg{\dMac-1} \min(\numroots, \monomialsleqdeg{\dMac-1}) \\
&+ \numroots^2 \monomialsdeg{\dMac}.
\end{align*}
\begin{itemize}
\item For fixed $\dimension$ and $\degree\rightarrow\infty$, a tight asymptotic bound is $O(\degree^{3\dimension+1}),$
which is the same as null space Macaulay SVD, but several lower-order terms are cheaper.
\item For fixed $\degree$ and $\dimension\rightarrow\infty$, a tight bound is $O(n^{-1/2}\beta^{3n}\alpha_{\beta}^{3n}).$
\end{itemize}
Note that the complexity for random combinations is the same for both direct Macaulay SVD and null space Macaulay SVD (compare the complexities in the previous section).

\subsection{Degree by Degree SVD}
\label{sec: temporal complexity: degree-by-degree}
Following the steps given in \cite{Telen4}, there are 3 steps for each iterative degree step of the degree-by degree-construction (see Section~\ref{sec: Degree by Degree}).

\begin{enumerate}
  \item Multiply $B_k$ by $\nullSpace{k}$. This is $\MatMult(\nullitymac{k}, \monomialsleqdeg{k}, T_{k+1})$.
  \item Find the kernel of $\begin{bmatrix}
 A_k & B_k N_k
 \end{bmatrix}$. This is $\SVD(\nullitymac{k}+\monomialsdeg{k+1}, T_{k+1})$.
 \item Multiply $L_{k+1}$ by $\hat{N}_{k+1}$. This really just requires multiplying part of $L_{k+1}$ by $N_k$. This is $\MatMult(\nullitymac{k+1}, \nullitymac{k}, \monomialsleqdeg{k})$.
\end{enumerate}

\begin{lemma}
With fixed degree, variable dimension, a tight asymptotic bound of the degree-by-degree construction is the same as the tight asymptotic bound of the final step.
\end{lemma}
\begin{proof}
See Appendix \ref{appendix: degree-by-degree final step fixed deg}.
\end{proof}

\begin{lemma}
With fixed dimension, variable degree, a tight asymptotic bound of the degree-by-degree construction is $\dMac$ times the tight asymptotic bound of the final step.
\end{lemma}
\begin{proof}
See Appendix \ref{appdendix: degree-by-degree final step fixed dim}.

\end{proof}

So for the final step where $k+1 = \dMac$ the complexity is
\begin{align*}
    & \nullitymac{\dMac - 1} \monomialsleqdeg{\dMac-1} T_\dMac \\ &+
    (\nullitymac{\dMac - 1}+\monomialsdeg{\dMac}) T_\dMac \min(\nullitymac{\dMac - 1}+\monomialsdeg{\dMac}, T_\dMac) \\ &+
    \numroots  \nullitymac{\dMac - 1} \monomialsleqdeg{\dMac-1}
\end{align*}
By Proposition~\ref{lem:koszul-nullity} we have $\nullitymac{\dMac-1} = \nullitymac{\dMac} = \degree^n$. 
Combined with the previous results, this gives the following bounds:
\begin{itemize}
    \item 
 For fixed $\dimension$ and $\degree\rightarrow\infty$, a tight asymptotic bound for the complexity of the degree-by-degree SVD method is
$O(\degree^{3\dimension}).$
\item For fixed $\degree$ and $\dimension\rightarrow\infty$, a tight asymptotic bound is
$O(n^{-1/2}\beta^{3n}\alpha_{\beta}^{3n}).$
\end{itemize}

\subsection{Complexity at low degree and dimension}
While the asymptotic bounds above give insight into the behavior of the algorithm when the degree $\beta$ or the  dimension $n$ is large, the temporal complexity of these methods and the sheer number of roots for large degrees and large dimensions mean that in practice the algorithm will only be used when both dimension and degree are relatively small. To compare performance at these more practical levels, we can directly calculate the number of floating point operations (FLOPs) of all the steps of the algorithm without much simplification. Comparing the FLOPs for the simple (no speedups) null space construction and the degree-by-degree null space construction gives a better sense of the savings we actually expect to see in practice.   

Of course, as dimension increases, the number of FLOPs increases exponentially for both variants. When dimension is fixed and degree varies, we see more interesting results, as shown in Figure~\ref{fig: dim_FLOPS}.

\begin{figure}
\centering
\includegraphics[width=.9\textwidth]{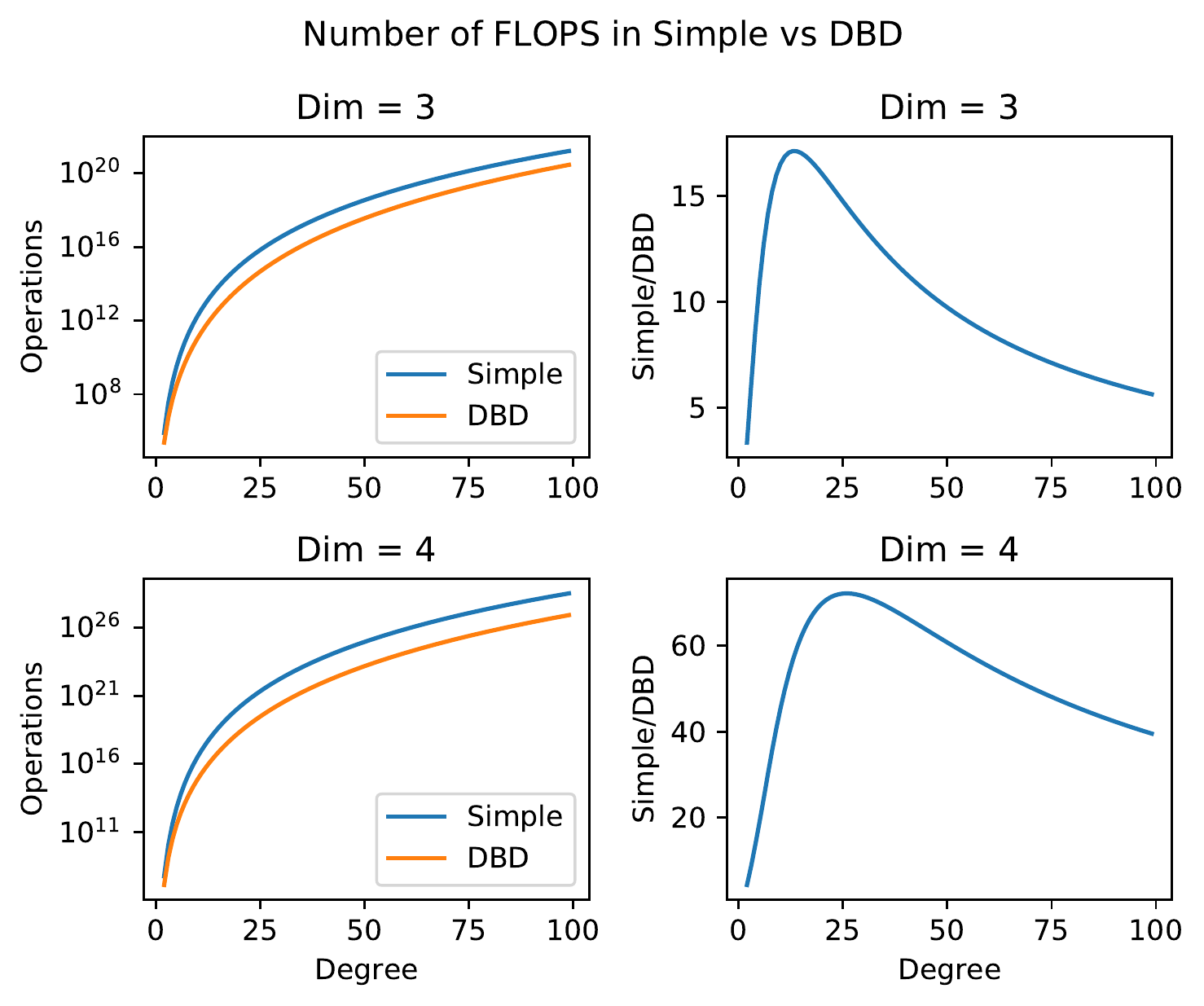}
\caption{Comparison of the FLOPs used with a simple (no speedups) null space reduction versus the degree-by-degree null space reduction.
The top and bottom rows show results in dimensions three and four, respectively. The panels on the left show the total number of FLOPs required for the two constructions. The panels on the right show the ratio of operations between  the simple construction and the  degree-by-degree construction. Notice that the peak, where the savings of degree-by-degree is most significant, moves to the right as dimension increases.}
\label{fig: dim_FLOPS}
\end{figure}

In dimension three the savings at degrees two and three are minimal, but by degree $10$ the difference is substantial. The biggest savings are roughly between degrees $10$ to $50$. As dimension increases, the range of degrees with the most savings moves upward. So although degree-by-degree requires fewer FLOPs than the simple construction, the amount of savings varies significantly with the degree and dimension.

\section{Numerical Stability}
\label{sec: Numerical Stability}

Unfortunately, the methods described above for polynomial rootfinding are unstable. This can be seen from the following quadratic system from \cite{Noferini}, which, following Townsend, we refer to as the \emph{devastating example}:

\[
\left(
\begin{array}{c}
    p_1(x_1,\ldots,x_n)\\
    \vdots \\ 
    p_n(x_1,\ldots,x_n)\\
\end{array}
\right)
=
\left(
\begin{array}{c}
    x_1^2\\
    \vdots\\
    x_n^2\\
\end{array}
\right)
+ \varepsilon Q
\left(
\begin{array}{c}
    x_1\\
    \vdots\\
    x_n\\
\end{array}
\right)
\]
where $Q$ is any unitary matrix and $\varepsilon> 0$ is small. 

Recall that the absolute condition number of a simple root $\mathbf z$ of $f:\mathbb{R}^n\mapsto\mathbb{R}^n$ is
    $$\kappa(\mathbf z,f) = \left\lVert{Df(\mathbf z)^{-1}}\right\rVert,$$
where $Df$ is the Jacobian of $f$ \cite[Proposition 14.1]{cucker},
and that the condition number of a simple eigenvalue $\lambda$ of matrix with left and right eigenvectors $\mathbf u$ and $\mathbf v$, respectively, is
    $$\kappa(\lambda,A) = \frac{\left\lVert{\mathbf u}\right\rVert \left\lVert{\mathbf v}\right\rVert}{|\mathbf u\hermitian\mathbf v|}$$
\cite[p.359]{Golub}.
At $\mathbf x^*=[0,\dots,0]\transpose$, the Jacobian of the devastating system is $J(\mathbf x^*)=\varepsilon Q$, so the condition number of the root at $\mathbf x^*$ is $\norm{J(\mathbf x^*)^{-1}}=\varepsilon^{-1}\norm{Q^{-1}}=\varepsilon^{-1}$.  However, as discussed below, the condition number of the corresponding eigenvalue of the M\"{o}ller--Stetter matrix computed using the SVD, QRP, or LQ methods described above is $\kappa(\lambda,M_{x_i}) = \Omega(\varepsilon^{-n})$, i.e., asymptotically of order at least $\varepsilon^{-n}$. This shows that the condition number of the eigenvalue may grow exponentially with dimension even though the condition number of the root is constant in dimension. If the algorithm were backwards stable, relative forward error would necessarily be $O(\kappa \epsilon_m)$ where $\epsilon_m$ is unit roundoff \cite[p. 111]{trefethen97}. This example displays forward error with the behavior of $O(\kappa^n \epsilon_m)$.

\newcommand{\cof}{\text{cof}}
\newcommand{\minor}{\text{minor}}
\newcommand{\SB}{{\mathbf B}}
\newcommand{\hatSB}{\mathcal B_{\text{QRP}}}
\newcommand{\hatBLQ}{\mathcal B_{\text{LQ}}}
\newcommand{\B}{\mathbf{B}_{\text{SVD}}}
\newcommand{\hatB}{\mathcal B_{\text{SVD}}}
We now discuss why the eigenvalue problem is this ill-conditioned. We begin by proving the form of the abstract eigenpolynomial of the operator $m_{x_j}:\quotientalgebra\rightarrow \quotientalgebra$.
\begin{lemma}
For $j=1,\ldots,n$, the eigenpolynomial associated with $\lambda = 0$ of the operator $m_{x_i}:\quotientalgebra\rightarrow \quotientalgebra$ is $$q = \sum_{\iota\subseteq \mathbf I} \det (\varepsilon Q_\iota) \prod_{k\in\iota} x_k$$
where $\mathbf I = \{1,\ldots,n\}$ denotes the set of possible row and column indices and $Q_\iota$ denotes the matrix formed by removing rows and columns in $\iota\subseteq \mathbf I$ from $Q$. 
\label{lemma: q is of this form}
\end{lemma}
\begin{proof}
First we show that $q\not \equiv 0 \pmod {\I}$. We proceed by contradiction. If $q \equiv 0 \pmod {\I}$, then $q$ evaluates to zero at all of the common roots of the generators of $\I$. But $q(\mathbf 0) = \det(\varepsilon Q)\neq 0$.

We now show that for $i=1,\ldots,n$, $x_iq \equiv 0 \pmod {\I}$. To do so, we fix $i$ and claim that
$$x_iq = \sum_{j=1}^np_j\left(\sum_{\iota\subseteq \mathbf I \setminus\{i,j\}}\cof_{ji}(\varepsilon Q_\iota)\prod_{k\in\iota}x_k\right)$$ where $\cof_{ji}(\varepsilon Q_\iota)$ denotes the cofactor of $\varepsilon Q$ obtained by removing rows $\iota\cup\{j\}$ and columns $\iota\cup\{i\}$ from $Q$. This is straightforward to prove, though algebraically tedious.
\end{proof}

The methods discussed in this paper choose bases in which the representation of $q$ leads to an ill-conditioned eigenproblem. 
To see this, we need the following lemma.
\begin{lemma}
Let $\SB$ be the standard basis for $\mathbb C[x_1,\ldots,x_n;d-1]$ and let $\hatSB = \{\prod_{k\in\iota} x_k: \iota \subseteq\mathbf I\} \subseteq \SB$. For a coefficient of a monomial in $\hatSB$ to appear in a polynomial $s\in\I$, that coefficient must be $O(\varepsilon)$.
\label{lemma: small coeffs in bqrp}
\end{lemma}
\begin{proof}
Let $s\in\I$. Then there exist polynomials $s_1,\ldots,s_n\in\mathbb C[x_1,\ldots,x_n]$ such that $$s = \sum_{i=1}^n s_i\left(x_i^2 + \varepsilon \sum_{j=1}^n q_{ij}x_j\right) = \left(\sum_{i=1}^n s_i x_i^2\right) + \varepsilon \sum_{i=1}^n\sum_{j=1}^n q_{ij}s_i x_j.$$ The conclusion follows.
\end{proof}
For simplicity, we order $\hatSB$ so that the monomial $1$ is at the end, and order $\SB$ so that the monomials in $\hatSB$ appear last. We now show that QRP, SVD and LQ methods all give ill-conditioned eigenproblems.
\begin{theorem}
When $M_{x_i}$ is constructed using the direct Macaulay QRP method, $\kappa(0,M_{x_i}) \geq \varepsilon^{-n}$.
\end{theorem}
\begin{proof}
By Lemma~\ref{lemma: small coeffs in bqrp}, the $\hatSB$ columns in $\Mac_3$ have entries that are $O(\varepsilon)$. It is straightforward to verify that QR with pivoting will never choose a $\hatSB$ column to be the pivot columns, so the QRP method will choose $\hatSB$ as the basis.
By (\ref{eq:left eigenvectors of MS matrix}), the left eigenvector associated with $\lambda=0$ is $\mathbf u = [0,\ldots,0,1]\transpose$, and the right eigenvector is $$\mathbf v = [q]_{\hatSB} = \begin{bmatrix}
1 \\  \varepsilon \det Q_{\{1,\ldots,n-1\}} \\ \vdots \\ \varepsilon^n
\end{bmatrix}.$$ It follows that
\begin{align*}
    \kappa(0,M_{x_i}) 
    = \frac{\left\lVert{\mathbf u}\right\rVert \left\lVert{\mathbf v}\right\rVert}{|\mathbf u\hermitian\mathbf v|}
    = \frac{\sqrt{1 + \ldots + \varepsilon^{2n}}}{\varepsilon^n}
    \geq \frac{1}{\varepsilon^n}.
\end{align*}
\end{proof}
\begin{theorem}
\label{theorem:svd lq condition numbers}
Using the direct Macaulay method with an SVD factorization results in an eigenvalue condition number of $\Omega(\varepsilon^{-n})$ if the SVD of $\Mac_3$ is computed via Golub-Kahan, LHC or three-step bidiagonalization followed by the Golub-Kahan diagonalization step.
Using an LQ factorization results in an eigenvalue condition number of $\Omega(\varepsilon^{-n})$ if the LQ factorization of $\Mac_3$ is computed via Householder QR.
\end{theorem}
\begin{proof}
We prove the result for the SVD method. An analogous proof gives the same result for the LQ method. First observe that the last column of $\Mac_3$ (i.e. the column for $1\in \SB$) is all zero because each row in $\Mac$ is a monomial multiple of some $p_i$, and no $p_i$ has a constant term.
It is straightforward to see that performing the SVD using the standard methods on $\Mac_3$ will result in a $V$ of the form $$V = \begin{bmatrix}& & & 0 \\ & \hat V & & \vdots \\ & & & 0\\ 0 & \cdots & 0 & 1\end{bmatrix}.$$

The matrix $V$ represents the basis transition matrix, so the final column of $V$ being of the form $[0, \ldots, 0, 1]\transpose$ means that the last element in the basis $\hatB$ which the SVD method chooses for $\quotientalgebra$ includes the monomial $1$ as the last element. Since $V$ is unitary, no other element in $\hatB$ has a constant term, so by (\ref{eq:left eigenvectors of MS matrix}), $\mathbf u = [0, \ldots, 0, 1]\hermitian$. The last entry in $\mathbf v=[q]_{\hatB}$ is $\varepsilon^n$, the constant term in $q$. Thus 
\begin{align*}
    \kappa(0,M_{x_i}) 
    = \frac{\left\lVert{[q]_{\hatB}}\right\rVert}{\varepsilon^n},
\end{align*}
so it suffices to show that $\left\lVert{[q]_{\hatB}}\right\rVert= \Omega(1)$ as $\varepsilon\rightarrow0$.

Partition $V$ so that $$V = \begin{bmatrix}V_1 & V_2 \\V_3 & V_4 \end{bmatrix}$$ and $V_4$ is $r \times r$. If $\B$ is the basis for $\mathbb C[x_1,\ldots,x_n;d-1]$ represented by the columns of $V$, then $[q]_{\B} = V\hermitian[q]_{\SB}$. When considered as an element of the quotient algebra $\quotientalgebra$,
$$[q]_{\hatB} = \begin{bmatrix} 0 & I\end{bmatrix}\begin{bmatrix}V_1\hermitian & V_3\hermitian \\V_2\hermitian & V_4\hermitian \end{bmatrix}\begin{bmatrix}\mathbf 0 \\ [q]_{\hatSB}  \end{bmatrix} = V_4\hermitian [q]_{\hatSB}.$$

Consider the rows in $\begin{bmatrix}V_1\hermitian & V_3\hermitian\end{bmatrix}$. Each consists of the coefficients in the basis $\SB$ of a polynomial in $\I$. By Lemma~\ref{lemma: small coeffs in bqrp}, for a monomial in $\hatSB$ to appear with a nonzero coeffieint, that coefficient must scale with $\varepsilon$, so $V_3\hermitian = \varepsilon \tilde V_3\hermitian$ for some matrix $\tilde V_3$ that is independent of $\varepsilon$. 
Because $V$ is unitary, $I = \varepsilon^2 \tilde V_3 \tilde V_3\hermitian + V_4 V_4\hermitian$.
Therefore
\begin{align*}\left\lVert{[q]_{\hatB}}\right\rVert^2 
&= [q]_{\hatSB}\hermitian \left(I - \varepsilon^2 \tilde V_3\tilde V_3\hermitian \right) [q]_{\hatSB} \\
&= \left\lVert{[q]_{\hatSB}}\right\rVert^2 -\varepsilon^2 \left\lVert{\tilde V_3\hermitian[q]_{\hatSB}}\right\rVert^2  \\
&= \Omega(1)
\end{align*}
as desired.
\end{proof}

The above proofs can also be extended to show that SVD/LQ/QRP nullspace methods also result in poor conditioning for the devastating example. We do not present these proofs here, but they follow naturally from the same ideas about choosing orthonormal bases for $\quotientalgebra$ that include $1$ and whose orthogonal complement is in the ideal. 

The devastating example suggests that choosing non-orthogonal bases or preconditioning the Macaulay matrix could improve the performance of the method. For example, if one multiplied the $1$'s column in the Macaulay matrix by $\varepsilon^n$ and then divided the $1$'s column in $V$ by $\varepsilon^n$, the algorithm would effectively choose $\varepsilon^n$ instead of $1$ to be in $\hatB$. Then the right eigenvector becomes $\mathbf v = [q]_{\mathcal B} = [\ldots, 1]\transpose,$
and the condition number is $\kappa(\lambda,M_{x_j}) \approx \sqrt{2}$ regardless of dimension. Unfortunately, it is difficult to see exactly how to do this rescaling in a general way that avoids conditioning problems. 

\FloatBarrier

\section{Numerical Experiments}
\label{sec: Numerical Experiments}

\par We ran numerical experiments to compare the speed and accuracy of these methods and their variants on several different types of systems. The M\"{o}ller--Stetter methods presented in this paper appear to perform well in practice on most low-dimensional problems of relatively small degree. 

\subsection{QRP and SVD Direct Reduction Comparison}
\label{sec: SVD vs QRP}

\begin{figure}[tbhp]
\centering
\includegraphics[width=0.9\textwidth]{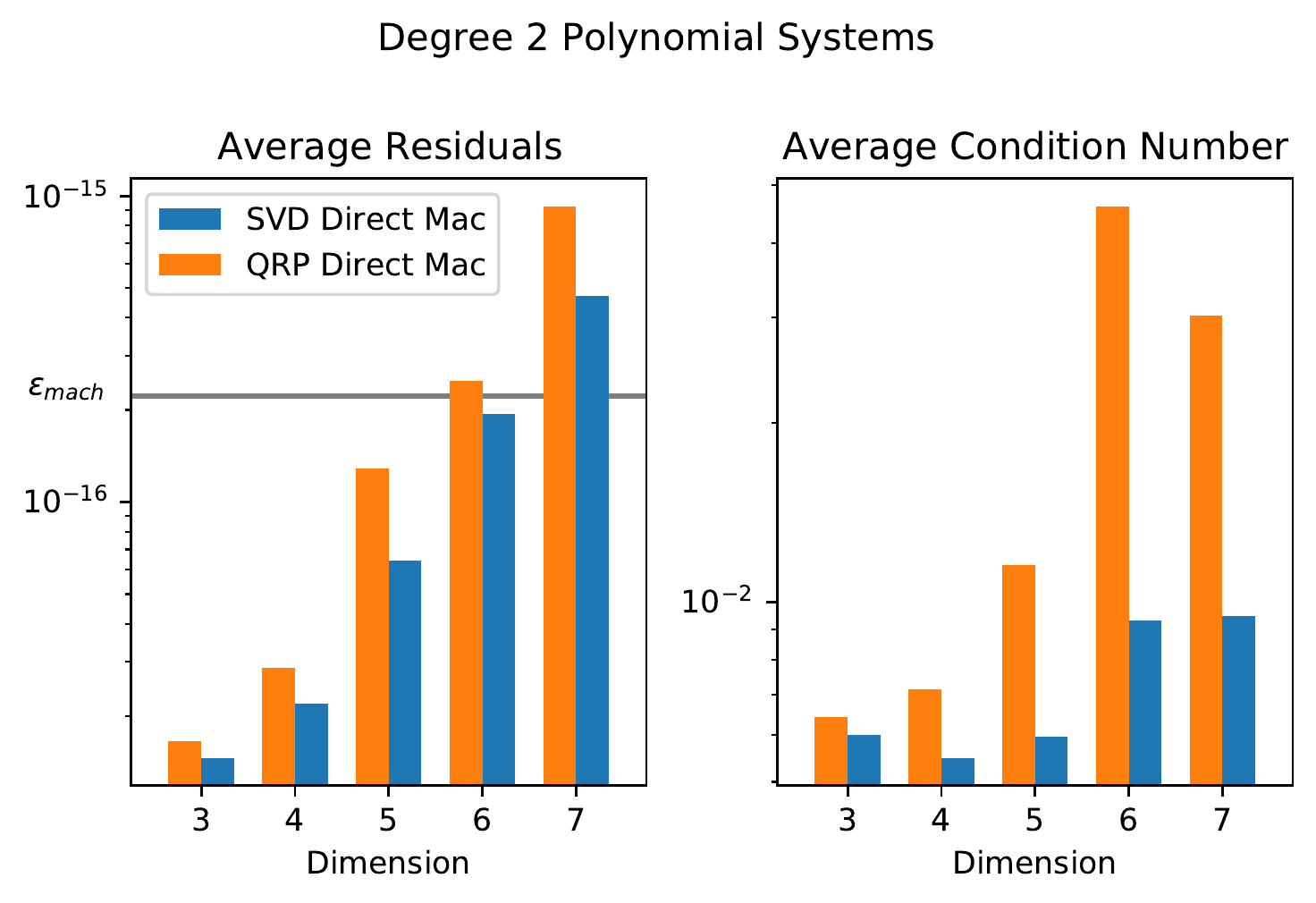}
\caption{Log-scale plots comparing  the Direct Macaulay QRP and SVD methods in terms of average residuals (left panel) and Macaulay condition numbers (right panel) for random, dense polynomials (in the power basis) of total degree $2$  over dimensions $3$ through $7$. Both the average residuals and the average eigenvalue condition number for the SVD method are smaller than for the QRP method. Similar improvements of SVD over QRP can be observed for a fixed dimension and varying degree.}
\label{fig:QRPvsSVD_res/eigs}
\end{figure}

\par We ran numerical experiments on random, dense polynomial systems with coefficients drawn from the standard normal distribution in both the power basis and Chebyshev basis of varying degree 
and dimension to compare the SVD method to the QRP method when reducing the Macaulay matrix directly (as opposed to using a null space method). We compared the average residuals of the roots, the average eigenvalue condition number, and computation time.

\par In general, we observed that systems solved using the SVD variant of direct Macaulay reduction had smaller average residuals and smaller condition numbers for eigenvalues than the same systems solved with the QRP method, for all degrees and dimensions. As the dimension increases, the improvements become more apparent (see Figure~\ref{fig:QRPvsSVD_res/eigs}). Surprisingly, the overall computation time was very similar between QRP and SVD although computing the SVD is generally more expensive.

\par In addition to random polynomial tests, we also compared the methods using several specific examples from Chebfun2's rootfinding test suite \cite{chebsuite}. Not all of the functions in the test suite are polynomials, so we ran the methods on high-degree Chebyshev polynomial interpolants. In many cases, the Macaulay matrix was too poorly conditioned for either method to work. This is not surprising since these systems are difficult by design to test the robustness of Chebfun2's numerical root finder, which utilizes subdivision to make subproblems that are more manageable. However, for the tests that were able to complete with these Chebyshev interpolants, the SVD method was faster than the QRP method. The maximum residuals for the SVD method were also better or the same for the QRP method most of the time.

\subsection{Null Space and Macaulay Method Comparison} 
\label{sec: compare methods}
Similar to the tests we ran above comparing the different methods to reduce the Macaulay matrix directly (as opposed to the null space), we ran experiments on random, dense polynomial systems in the power basis with coefficients drawn from the standard normal distribution of varying degree and dimension to compare the class of Macaulay null space reduction methods with the class of direct Macaulay reduction methods. We found that using the SVD variant of each method tends to yield the best results in terms of residuals with a similar trend as that apparent in Figure~\ref{fig:QRPvsSVD_res/eigs}.

\begin{figure}[tbhp]
\centering
\includegraphics[width=.9\textwidth]{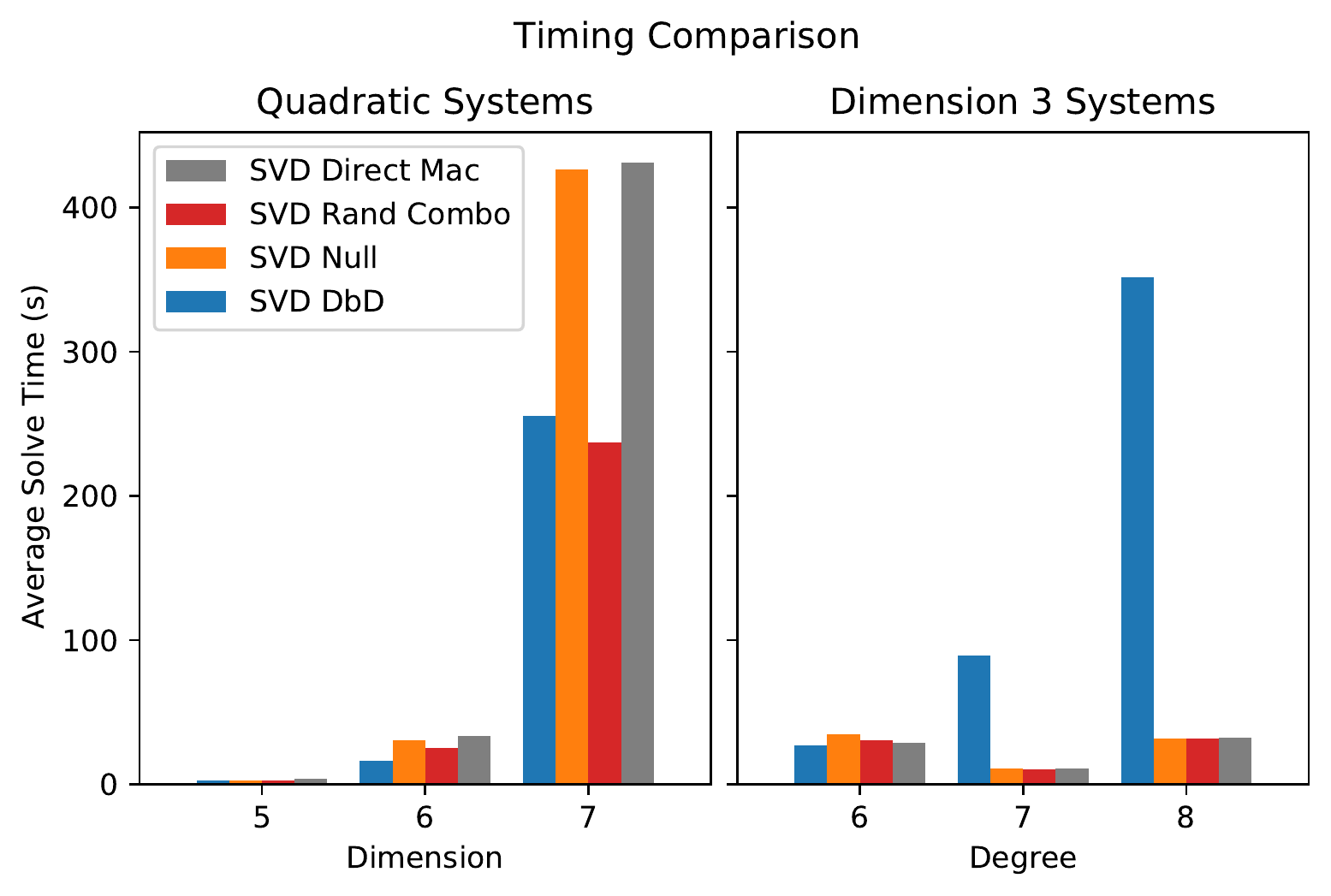}
\caption{Average solve time for random, dense, polynomial systems using the SVD variants of direct Macaulay, Macaulay null space reduction, and the degree-by-degree method. On the left, we have quadratic systems of varying dimension. On the right, we have Dimension 3 polynomial systems of varying degree. Although the QRP method generally performs faster, the SVD method gives more accurate results (see Figure~\ref{fig:QRPvsSVD_res/eigs}). The results for systems of dimensions and degrees not visible on these plots are solved quickly enough to make their inclusion unhelpful for comparison using a linear scale.}
\label{fig: SVD_comp_times_lin}
\end{figure}

\begin{figure}[tbhp]
\centering
\includegraphics[width=.9\textwidth]{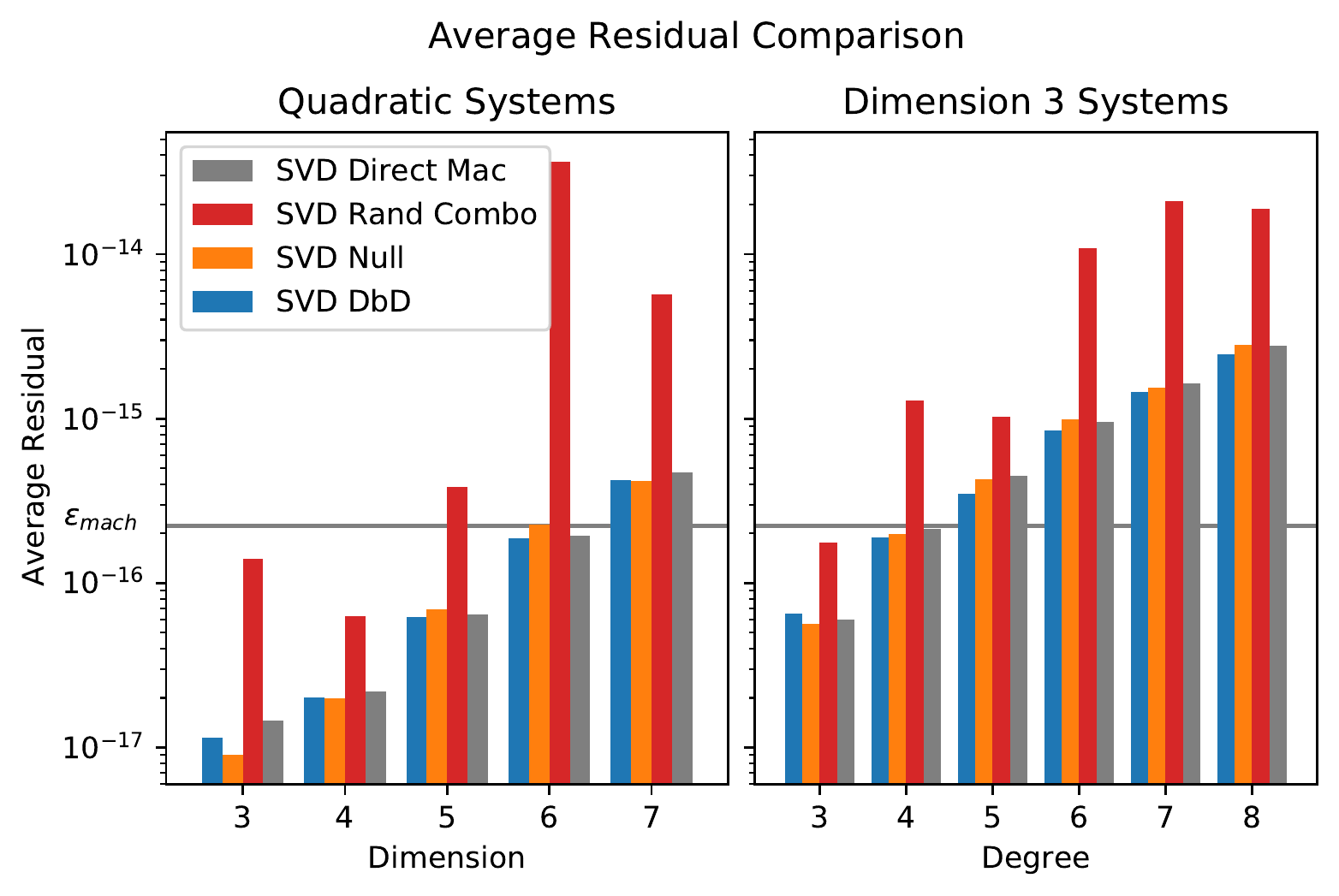}
\caption{Average residuals for random, dense, polynomial systems using the SVD variants of direct Macaulay, Macaulay null space reduction, and the degree-by-degree methods. On the left, we have quadratic systems of varying dimension. On the right, we have Dimension 3 polynomial systems of varying degree. Comparing the residuals of QRP yields a similar trend, but the SVD consistently outperforms QRP in terms of accuracy. }
\label{fig: SVD_res_comp_log}
\end{figure}

\par In Figure~\ref{fig: SVD_comp_times_lin}, one can see that the degree-by-degree method provides a significant speed advantage for low-degree systems in high dimensions, but for a fixed dimension, it becomes more computationally expensive as degree increases while only providing slightly better residuals (see Figure~\ref{fig: SVD_res_comp_log}). This is surprising given that in Section~\ref{sec: temporal complexity: degree-by-degree} we computed the asymptotic temporal complexity of the degree-by-degree method to be $\degree^2$ cheaper than the direct Macaulay reduction for fixed $\dimension$. Additionally, we see that using the random combinations method with the direct Macaulay reduction causes the residuals to be worse by a factor of 10 to 100. The speedup gained using random combinations as dimension increases does not make it much faster than the degree-by-degree method. As degree increases for a fixed dimension, it appears that random combinations provides an insignificant speed boost while giving worse residuals. This is seems to agree with the temporal complexity computed in Section~\ref{sec: temporal complexity: rand combos}, where the ratio of the asymptotic complexity of direct Macaulay reduction and the asymptotic complexity of random combinations is a factor of dimension alone.

\par Overall, the results indicate that for low-degree, high-dimensional problems, the degree-by-degree method is the fastest while also being among the most accurate  (see Figures~\ref{fig: SVD_comp_times_lin} and~\ref{fig: SVD_res_comp_log}). For high-degree, low-dimensional problems, it appears that the SVD null space method and the SVD direct Macaulay method presented in Section~\ref{sec: Direct Macaulay Reduction} performed with similar speed and accuracy, which matches our intuition given that their temporal complexities are shown to be equal in sections \ref{sec: temporal complexity: Direct Macaulay} and \ref{sec: temporal complexity: null space}. See Figure~\ref{fig: SVD_comp_times_lin}.

\FloatBarrier

\subsection{The Devastating Example}
To explore the frequency of behavior like the devastating example, we define the \textit{conditioning ratio} 
for a M\"{o}ller--Stetter eigenproblem. This definition is inspired by Trefethen and Bau's analysis of the stability of Gaussian elimination \cite[p. 164]{trefethen97}. 
We define the conditioning ratio for a M\"oller--Stetter eigenproblem for an eigenvalue $\lambda$ corresponding to a root $z$ to be
    $$CR(\lambda,z,f,M_g) = \frac{\kappa(\lambda,M_g)}{\kappa(z,f)}.$$
In practice, we use the method of \cite{VanLoan} to compute the eigenvalue condition number.
The base-10 logarithm of the conditioning ratio measures how many additional digits of precision may be lost when converting the root-finding problem into an eigenproblem. We also define the \textit{growth rate} of the conditioning ratios of a family of problems to be the value $g$ such that the conditioning ratio is approximately $C(1+g)^n$ for some constant $C$. The growth rate can be numerically estimated via $g = b^s - 1$ where $s$ is the slope of the line of best fit to the base-$b$ logarithm of computed conditioning ratios. The conditioning ratio of the devastating example is $\Omega(\varepsilon^{1-n})$ with a growth rate of $g = \varepsilon^{-1}-1$, and numerical computation is consistent with these theoretical values. In particular, Figure~\ref{fig: Growth Factors Dev Rand} shows that the slope of the base-10 log of the conditioning ratios as dimension increases matches the theoretical slope of $-\log_{10} \varepsilon$. 

As seen in Figure~\ref{fig: Growth Factors Dev Rand}, random polynomials behave much better than the devastating example, even when $\varepsilon$ is relatively large (e.g. $10^{-1}$), which corresponds to a ``not-so-bad" devastating example. Although the conditioning ratio still appears to grow exponentially with dimension, the slope of the line of best fit shows that the growth is much slower. This suggests that, in many cases, the M\"{o}ller--Stetter methods can still give accurate results in low-dimension problems despite being numerically unstable on some special examples.

\begin{figure}[tbhp]
\centering
\includegraphics[width=0.9\textwidth]{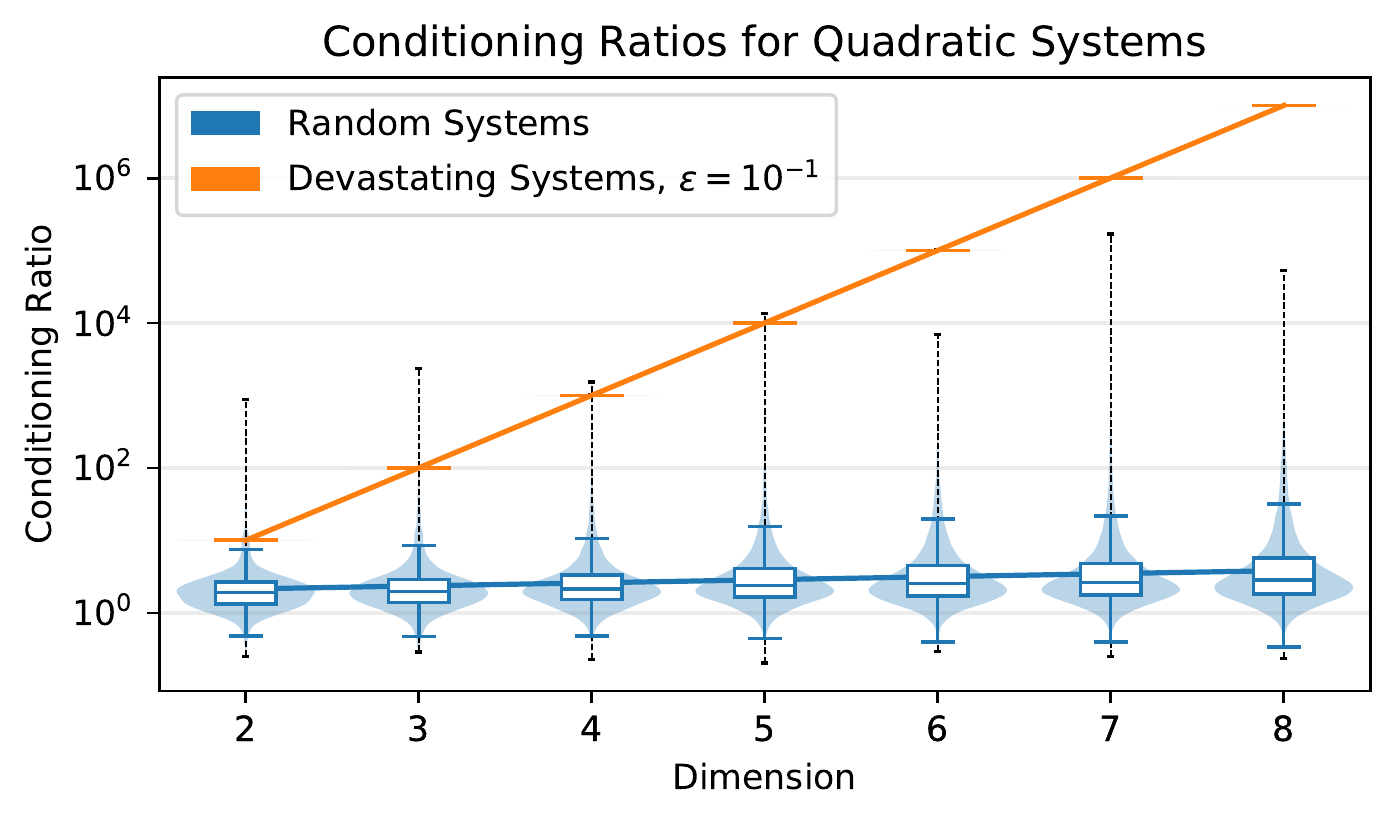}
\caption{Numerically calculated conditioning ratios for devastating and random quadratic systems solved using the direct Macaulay SVD method. The conditioning ratios of random systems show very slow exponential growth in dimension compared to the devastating example. Orange: Line of best fit for conditioning ratios of $n$ dimensional devastating systems with a randomly chosen $Q$. All of the computed conditioning ratios were within $0.015\%$ of the theoretical value, and the computed growth rate was $9.001$. Blue: Random systems of quadratic polynomials with coefficients drawn from the standard normal distribution. The violin and box plots show the distributions of the conditioning ratios of theses systems. The dotted black lines represent the tail ends of these distributions out to the most extreme observed conditioning ratios. The line of best fit to the base-10 logarithm of the conditioning ratios is also shown, with a growth rate of $g \approx 0.102$.}
\label{fig: Growth Factors Dev Rand}
\end{figure}

Perturbation of the devastating example seems to slow the exponential increase in conditioning ratio. To explore this numerically, we perturb devastating systems by adding a random quadratic polynomial with coefficients drawn from a normal distribution with standard deviation $\delta$. As seen in Figure~\ref{fig: Growth Factors Perturb}, larger perturbations correspond to slower exponential growth in conditioning ratio. This behavior occurs because perturbation of the problem creates a dense system, which opens up more choices for the basis of $\quotientalgebra$, and experimentally many of these newly available bases correspond to better conditioned eigenproblems.

\begin{figure}[tbhp]
\centering
\includegraphics[width=0.9\textwidth]{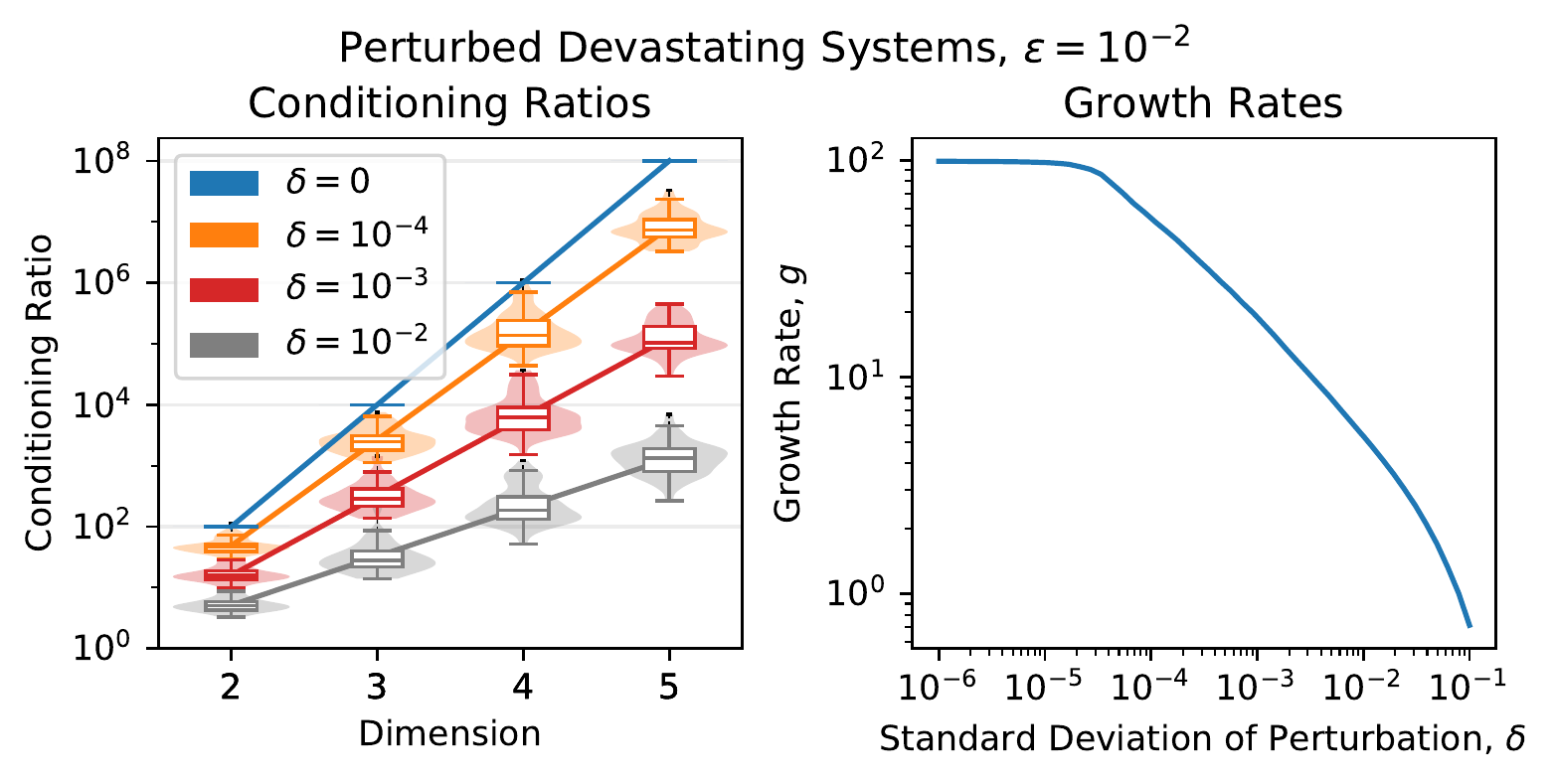}
\caption{Numerically calculated conditioning ratios for devastating systems with $\varepsilon = 10^{-2}$ that are perturbed by adding a value drawn from $\mathcal N(0,\delta^2)$ to each coefficient. The growth rate decreases with larger perturbations. For $\delta = 0,10^{-4},10^{-3},10^{-2}$, the computed growths rates are $g \approx 99.000,54.081,18.886$ and $5.341$, respectively (left). The growth rate decreases for $\delta > 10^{-5}$ (right).}
\label{fig: Growth Factors Perturb}
\end{figure}

The devastating example is close to a system with a very high multiplicity root. All of the roots of the system scale linearly with $\varepsilon$, so when $\varepsilon=0$, there is a root of order $2^n$. To better explore the behavior of M\"{o}ller--Stetter methods when roots are almost high multiplicity, we generate random systems of special quadratic polynomials for which it is easy to control the location of some of the roots. We then examine the behavior of the conditioning ratios when those roots are forced to be close together. In particular, we consider systems where each polynomial is of the form
$$f(\mathbf x) = 1 - \sum_{j=1}^n a_j(x_j - c_j)^2.$$
In two dimensions, this results in zero loci that are either hyperbolas or ellipses with axes along the coordinate directions. To force $\root_1,\ldots,\root_n$ to be roots with $\root_i = (r_{i1},\ldots,r_{in})$ for all $i$, one can simply choose a center $\mathbf c = (c_1,\ldots,c_n)$ for the generalized conic and solve the linear system 
$$\begin{bmatrix}
(r_{11} - c_1)^2       &   (r_{12} - c_2)^2       & \dots     &   (r_{1n} - c_n)^2       \\
(r_{21} - c_1)^2       &   (r_{22} - c_2)^2       & \dots     &   (r_{2n} - c_n)^2       \\
\vdots  &  \vdots   &   \vdots  &   \vdots  \\   
(r_{n1} - c_1)^2       &   (r_{n2} - c_2)^2       & \dots     &   (r_{nn} - c_n)^2       \\
\end{bmatrix}
\begin{bmatrix}
a_1    \\
a_2     \\
\vdots  \\
a_n     \\
\end{bmatrix}=
\begin{bmatrix}
1    \\
1     \\
\vdots  \\
1     \\
\end{bmatrix}.$$
Repeating this process with $n$ different centers gives $n$ quadratics that share roots at $\root_1,\ldots,\root_n$. 
When $\root_1,\ldots,\root_n$ are forced to be slight perturbations of each other, the conditioning ratios increase rapidly as more and more roots are forced to be close together, as shown in Figure~\ref{fig: Growth Factors Mutliplicity}. It appears that having many nearby roots affects the eigenvalue condition number much more than the root condition number; see Figure~\ref{fig: Growth Factors Eig v Root Cond}. This suggests that at least part of what makes the devastating example problematic for M\"{o}ller--Stetter methods is that many roots are close together. 

In short, although the M\"{o}ller--Stetter methods are known to be unstable in some cases, in practice they seem to give accurate answers for low-dimensional, sufficiently well-behaved problems. Issues can arise when roots are nearly high multiplicity, but double roots do not seem to be particularly problematic. These challenges are unlikely to occur as long as the roots are sufficiently separated. 

\begin{figure}[tbhp]
\centering
\includegraphics[width=0.9\textwidth]{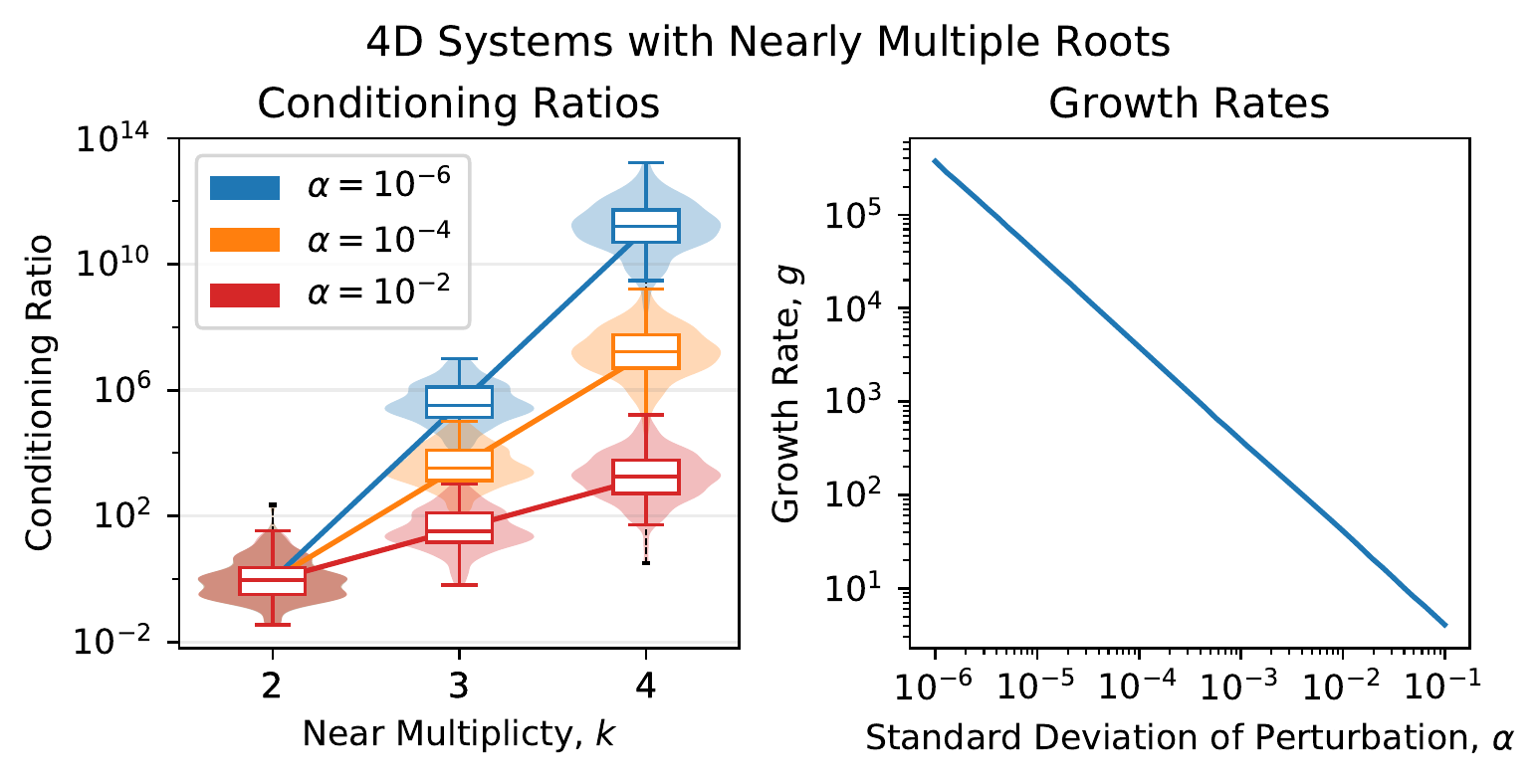}
\caption{Like the devastating example, which nearly has a multiplicity $2^n$ root at the origin, four dimensional systems that nearly have a multiplicity $k$ root also have poor conditioning ratios. The conditioning ratios appear to grow exponentially with $k$. Systems were generated to have a randomly chosen primary root $\root_1$, and nearby roots $\root_2,\ldots,\root_k$ that are random perturbations of $\root_1$ in every coordinate direction by values drawn from $\mathcal N(0,\alpha^2)$. Thus, $\root_2,\ldots,\root_k$ scale approximately linearly towards the primary root with $\alpha$. The growth rate (now in $k$, rather than $n$) depends heavily on $\alpha$.} 
\label{fig: Growth Factors Mutliplicity}
\end{figure}

\begin{figure}[tbhp]
\centering
\includegraphics[width=0.9\textwidth]{Figures/mult_eig_v_root_cond.pdf}
\caption{In four dimensional systems that nearly have a multiple root, the condition number of the eigenvalue increases much more rapidly than the condition number of the root as number $k$ of nearby roots increases. This effect is even more dramatic as the distance between roots decreases. Here, this distance is $O(\alpha)$.}
\label{fig: Growth Factors Eig v Root Cond}
\end{figure}

\FloatBarrier

\section{Conclusion and Future Work}
\par The improvements to the QRP method outlined in this paper were originally motivated by the idea to create an algorithm to find the real roots of a multidimensional system of arbitrary smooth functions. Working in the Chebyshev basis not only gives better conditioning of the real roots, but it also allows for this method to be seamlessly integrated into a Chebyshev proxy method as described in \cite{BoydBook}. Such a method consists of subdividing an original search domain into smaller subdomains until all of the functions can be accurately approximated with low-degree Chebyshev polynomials and then uses another method, like the ones described in this paper, to find the roots of the resulting systems of Chebyshev polynomials.

\par We plan to address under what conditions this solver
performs optimally in such a Chebyshev proxy method in a future paper. We are particularly interested in how well it performs when compared to other solvers (such as the B\'{e}zout Resultant, which the MATLAB package Chebfun uses in 2 dimensions) and its potential to operate in dimensions as high as 5 or 6.

\par It would also be interesting to explore methods to optimally precondition the Macaulay matrix and rescale basis elements to avoid problems like the devastating example.

\FloatBarrier
\clearpage

\pagebreak

\appendix

\section{Temporal Complexity Proofs}
\label{appendix: complexity proofs}

\par In this section, we provide the proofs for the lemmas found in Section~\ref{sec: temporal complexity: variable bounds}, which are concisely summarized in table~\ref{table: complexity bounds}.

\subsection{Fixed Dimension, Varying Degree}
\label{appendix: Fixed dim}
\par This section's lemmas are for the situation where the dimension $n$ is fixed and the degree $\beta$ goes to infinity.

\begin{lemma} 
\label{appendix: monomialsleqdeg dMac fixed dim}
With fixed dimension, variable degree, the function $\degree^\dimension$ is a tight asymptotic bound for $\monomialsleqdeg{\dMac-1}$, $\monomialsleqdeg{\dMac}$, and $\polysdeg{\dMac}$. The function $\degree^{\dimension-1}$ is a tight asymptotic bound for $\monomialsdeg{\dMac}$.
\end{lemma}
\begin{proof}
By the definitions of $\monomialsleqdeg{k}$ and $d$, we have
$\monomialsleqdeg{\dMac-1} 
= \binom{\dimension\degree}{\dimension\degree-\dimension}.
$
Therefore
$$
\degreelim \degree^{-\dimension}\monomialsleqdeg{\dMac-1}
= \frac{1}{\dimension!} \degreelim \prod_{j=1}^\dimension \frac{\dimension\degree-\dimension + j}{\degree} \\
= \frac{\dimension^\dimension}{\dimension!}
.$$
Since $\dimension$ is assumed to be constant, this implies that $\monomialsleqdeg{\dMac-1}\sim \frac{\dimension^\dimension}{\dimension!}\degree^\dimension$ as $\degree\rightarrow\infty$.
Similar computations give
$
\degreelim \degree^{-\dimension}\monomialsleqdeg{\dMac}
=
\degreelim \degree^{1-\dimension}\monomialsdeg{\dMac}
= \frac{\dimension^{\dimension}}{\dimension!}
$
and
$\degreelim \degree^{-\dimension}\polysdeg{\dMac}
= \frac{(\dimension-1)^{\dimension}}{(\dimension-1)!}$.
The result follows.
\end{proof}

\begin{lemma} 
\label{appendix: Sum Bound Lemma}
Let $A_i(\degree)$ be on nondecreasing, positive sequence dependent on $\degree$ for $0 \leq i \leq K(\degree)$, where $K:\mathbb N \rightarrow \mathbb N$. If there is some $C > 0$ such that $A_{\halfFloor{K(\degree)}}(\degree) \geq C A_{K(\degree)}(\degree)$ for all $\degree$, then $\left(K(\degree)+1\right)A_{K(\degree)}(\degree)$ is a tight asymptotic bound for $\sum_{i=0}^{K(\degree)} A_i(\degree)$ for fixed $n$ as $\degree\rightarrow\infty$.
\end{lemma}
\begin{proof}
Observe that
\begin{align*}
1 \geq 
\frac{\sum_{i=0}^{K(\degree)} A_i(\degree)}{\left(K(\degree)+1\right)A_{K(\degree)}(\degree)}
&\geq
\frac{\sum_{i=\halfCeil{K(\degree)}}^{K(\degree)} A_i(\degree)}{2K(\degree)A_{K(\degree)}(\degree)} 
\\ &\geq
\frac{\frac{K(\degree)}{2} A_{\halfCeil{K(\degree)}}(\degree)}{2K(\degree)A_{K(\degree)}(\degree)} 
\\ &\geq
\frac{C}{4}.
\end{align*}
Thus $$\frac{C}{4}\leq\degreeliminf \frac{\sum_{i=0}^{K(\degree)} A_i(\degree)}{\left(K(\degree)+1\right)A_{K(\degree)}(\degree)}
\leq \degreelimsup \frac{\sum_{i=0}^{K(\degree)} A_i(\degree)}{\left(K(\degree)+1\right)A_{K(\degree)}(\degree)}
 \leq 1.$$
Therefore $\left(K(\degree)+1\right)A_{K(\degree)}(\degree)$ is a tight asymptotic bound for $\sum_{i=0}^{K(\degree)} A_i(\degree)$.
\end{proof}

\begin{lemma}
\label{appendix: poysinmac bound lemma}
With fixed dimension, variable degree, the function $\degree^{\dimension+1}$ is tight asymptotic bound for $\polysinmac{\dMac}$.
\end{lemma}

\begin{proof} We have
$\polysinmac{\dMac} 
= \dimension \sum_{i=\degree}^\dMac \monomialsleqdeg{i-\degree} = \dimension \sum_{i=0}^{\dMac-\degree} \monomialsleqdeg{i}$.
A computation similar to Lemma~\ref{appendix: monomialsleqdeg dMac fixed dim}
gives $(\dMac - \degree+1)\monomialsleqdeg{\dMac-\degree} \sim C\degree^{\dimension+1}$, so it suffices to show that $(\dMac - \degree+1)\monomialsleqdeg{\dMac-\degree}$ is a tight asymptotic bound for $\polysinmac{\dMac}$. By Lemma~\ref{appendix: Sum Bound Lemma}, this is true if $\monomialsleqdeg{\floor{\frac{\dMac-\degree}{2}}}\geq C\monomialsleqdeg{\dMac-\degree}$ for some $C>0$. We see that
\begin{align*}
\frac{V_{\dMac - \degree}}{V_{\halfFloor{\dMac - \degree}}} 
&=
\prod_{j=1}^{\dimension}\frac{\dimension\degree-\degree-\dimension+1+j}{\floor{\frac{\dimension\degree-\degree-\dimension+1}{2}} + j}
\\&\leq 
\prod_{j=1}^{\dimension}\frac{\dimension\degree-\degree+1}{\frac{1}{2}\left(\dimension\degree-\degree-\dimension+1\right)}
\\&\leq 6^\dimension.
\end{align*}
This last inequality follows because
$
\frac{\dimension\degree-\degree+1}{\dimension\degree-\degree-\dimension+1}
= 1 + 
\frac{\dimension}{(\dimension-1)(\degree-1)}
\leq 
3.
$
Therefore, $\degree^{\dimension+1}$ is a tight asymptotic bound.
\end{proof}

\begin{lemma}
\label{appdendix: Nullity ratio bound}
Let $\nullitymacappendix{k} = \nullitymac{k}$.
Then $\frac{\nullitymacappendix{k}}{\numroots} \geq \frac{1}{2n!}$ for any $k \geq \degree$.
\end{lemma}
\begin{proof}
The nullity of the Macaulay matrix is nondecreasing as the degree of the matrix increases, so it suffices to prove this for $k=\degree$. Note that when all the $\degree_i$ are the same, the nullity formula is $\sum_{j=0}^{\dimension}(-1)^j\binom{\dimension}{j}\binom{\dimension+\dMac-j\degree}{\dimension}$.

Observe that 
$\nullitymacappendix{\degree} =
\sum_{j=0}^{\dimension}(-1)^j\binom{\dimension}{j}\binom{\dimension+\degree-j\degree}{\dimension} = 
\binom{\dimension+\degree}{\dimension} - \dimension \geq
\frac{1}{2}\binom{\dimension+\degree}{\dimension}
$ because $\frac{1}{2}\binom{\dimension+\degree}{\dimension} \ge \frac{1}{2}\binom{\dimension+2}{\dimension} = \frac{1}{4}(\dimension+1)(\dimension+2) > \dimension$.
So
$$\frac{\nullitymacappendix{\degree}}{\numroots} \ge
\frac{\binom{\dimension+\degree}{\dimension}}{2\numroots} =
\frac{1}{2n!}\frac{\prod_{j=1}^\dimension(\degree+j)}{\degree^\dimension} \ge
\frac{1}{2n!}
$$
\end{proof}

\begin{lemma}
\label{appdendix: monomialsdeg ratio bound}
For fixed $\dimension$, $\monomialsdeg{\halfFloor{\dMac}} > C\monomialsdeg{\dMac}$.
\end{lemma}
\begin{proof}
Observe that
\begin{align*}
\frac{\monomialsdeg{\halfFloor{\dMac}}}{\monomialsdeg{\dMac}} 
&= 
\frac{\binom{\dimension+\halfFloor{\dMac}-1}{\dimension-1}}{\binom{\dimension+\dMac-1}{\dimension-1}} 
=
\frac{(\dimension+\halfFloor{\dMac}-1)!}{(\dimension+\dMac-1)!}\frac{\dMac!}{\halfFloor{\dMac}!} 
=
\prod_{k=\halfFloor{\dMac}+1}^{\dMac}\frac{k}{\dimension+k-1} 
\ge
(\frac{\halfFloor{\dMac}+1}{\dimension+\halfFloor{\dMac}})^{\halfCeil{\dMac}}
\\&=
(1-\frac{\dimension-1}{\dimension + \halfFloor{\dMac}})^{\halfCeil{\dMac}} 
\ge
(1-\frac{\dimension}{\halfFloor{\dimension\degree+\dimension+1}})^{\dimension\degree} 
\ge
(1-\frac{2}{\degree+1})^{\dimension\degree+\dimension} 
\ge
(1-\frac{2}{3})^{3\dimension} 
\\&=\frac{1}{27^{\dimension}}
\end{align*}
because $(1-\frac{2}{x})^x$ is an increasing function when $ x \ge 2$ and $\degree \ge 2$. 
\end{proof}

\begin{lemma}
\label{appdendix: monomialsleqdeg ratio bound}
For fixed $\dimension$, $\monomialsleqdeg{\halfFloor{\dMac}-1} > C\monomialsleqdeg{\dMac-1}$.
\end{lemma}
\begin{proof}
Similar to the previous lemma, we see that
$$
\frac{\monomialsleqdeg{\halfFloor{\dMac}-1}}{\monomialsleqdeg{\dMac-1}} = 
\frac{\binom{\dimension+\halfFloor{\dMac}-1}{\dimension}}{\binom{\dimension+\dMac-1}{\dimension}} =
\prod_{k=\halfFloor{\dMac}}^{\dMac-1}\frac{k}{\dimension+k} \ge
(\frac{\halfFloor{\dMac}}{\dimension+\halfFloor{\dMac}})^{\halfCeil{\dMac}} \ge
(1-\frac{2}{\degree+1})^{\dimension \degree+\dimension} \ge
\frac{1}{27^{\dimension}}.
$$
\end{proof}

\begin{lemma}
\label{appdendix: polysdeg ratio bound}
For fixed $\dimension$, $\polysdeg{\halfFloor{\dMac+\degree}} > C \polysdeg{\dMac}$.
\end{lemma}
\begin{proof}
Observe that
\begin{align*}
\frac{\polysdeg{\halfFloor{\dMac+\degree}}}{\polysdeg{\dMac}} 
&= 
\frac{\monomialsleqdeg{\halfFloor{\dMac-\degree}}}{\monomialsleqdeg{\dMac-\degree}} 
=
\prod_{k=\halfFloor{\dMac-\degree}+1}^{\dMac-\degree}\frac{k}{\dimension+k} 
= 
\prod_{k=\halfFloor{\dMac-\degree}+1}^{\dMac-\degree}(1-\frac{\dimension}{\dimension+k}) 
\\&\ge 
(1-\frac{\dimension}{\dimension+\halfFloor{\dMac-\degree}+1})^{\halfCeil{\dMac-\degree}} 
\ge
(1-\frac{2\dimension}{\dimension\degree+\dimension-\degree+2})^{\dimension\degree+\dimension-\degree+2}.
\end{align*}
For all $x \geq 2\dimension+1$, $(1-\frac{2\dimension}{x})^x \ge (1-\frac{2\dimension}{2\dimension+1})^{2\dimension+1}$ because $(1-\frac{2\dimension}{x})^x$ is an increasing function. Letting $x=\dimension\degree+\dimension-\degree+2$ gives a tight asymptotic bound of $$(1-\frac{2\dimension}{2\dimension+1})^{2\dimension+1} = \left(\frac{1}{2\dimension+1}\right)^{2\dimension+1}.$$
\end{proof}

\begin{lemma}
\label{appdendix: degree-by-degree final step fixed dim}
With fixed dimension, variable degree, a tight asymptotic bound of the degree-by-degree construction is $\dMac$ times the tight asymptotic bound of the final step.
\end{lemma}

\begin{proof}
Define $\DBDStep{k}$ to be the complexity of the step that results in the null space of the Macaulay Matrix of degree $k$. Then $$\DBDStep{k} = 
\nullitymacappendix{k-1}\monomialsleqdeg{k-1}\polysdeg{k} + 
(\nullitymacappendix{k-1}+\monomialsdeg{k}) \polysdeg{k} \min(\nullitymacappendix{k-1}+\monomialsdeg{k}, \polysdeg{k}) + 
\nullitymacappendix{k} \nullitymacappendix{k-1} \monomialsleqdeg{k-1}.
$$

The cost of the entire construction is $\sum_{k=\degree+1}^{\dMac}\DBDStep{k} = 
\sum_{k=0}^{\dMac-\degree-1}\DBDStep{k+\degree+1}$. Using \ref{appendix: Sum Bound Lemma}, it suffices to show that $\DBDStep{\halfFloor{\dMac-\degree-1}+\degree+1} > C \DBDStep{\dMac}$ for some $C > 0$.
From Lemma~\ref{appdendix: Nullity ratio bound} and the fact that 
$$\halfFloor{\dMac-\degree-1}+\degree = \halfFloor{\dMac+\degree-1} = \halfFloor{\dimension\degree -\dimension +\degree} \ge \halfFloor{2\degree} = \degree,$$ 
we have that $\nullitymacappendix{\halfFloor{\dMac-\degree-1}+\degree} > C_1\numroots$ and
$\nullitymacappendix{\halfFloor{\dMac-\degree-1}+\degree+1} > C_2\numroots$ for $C_1$, $C_2 > 0.$ \\
From Lemma~\ref{appdendix: monomialsdeg ratio bound} and $\monomialsdeg{i}$ being increasing, we have that 
$$\monomialsdeg{\halfFloor{\dMac-\degree-1}+\degree} > \monomialsdeg{\halfFloor{\dMac}} \ge C_3 \monomialsdeg{\dMac}$$ for $C_3 > 0$.
From Lemma~\ref{appdendix: monomialsleqdeg ratio bound} and $\monomialsleqdeg{i}$ being increasing, we have that 
$$\monomialsleqdeg{\halfFloor{\dMac-\degree-1}+\degree} > \monomialsleqdeg{\halfFloor{\dMac}-1} \ge C_4 \monomialsleqdeg{\dMac}$$ for $C_4 > 0$.
From Lemma~\ref{appdendix: polysdeg ratio bound} and $\polysdeg{i}$ being increasing, we have that 
$$\polysdeg{\halfFloor{\dMac-\degree-1}+\degree+1} \ge \polysdeg{\halfFloor{\dMac+\degree}} \ge C_5 \polysdeg{\dMac}$$ for $C_5 > 0$.
Let $C = \min(C_1, C_2, C_3, C_4, C_5)^3$.
Then 
\begin{align*}
\DBDStep{\halfFloor{\dMac-\degree-1}+\degree+1} 
\ge
(C_1\nullitymacappendix{\dMac-1}
+C_3\monomialsdeg{\dMac}) C_5\polysdeg{\dMac} \min(C_1\nullitymacappendix{\dMac-1}+C_3\monomialsdeg{\dMac}, C_5\polysdeg{\dMac}) 
\\+ 
C_2\numroots C_1\nullitymacappendix{\dMac-1} \monomialsleqdeg{k-1} 
+C_1\nullitymacappendix{\dMac-1}C_4\monomialsleqdeg{\dMac-1}C_5\polysdeg{\dMac} 
\ge C\DBDStep{\dMac}.
\end{align*}
\end{proof}

\subsection{Fixed Degree, Varying Dimension}
\label{appendix: Fixed deg}
\par This section's lemmas are for the situation where the dimension $n$ goes to infinity and the degree $\beta$ is fixed.

\begin{lemma}
\label{appendix: monomialsleqdeg deg dMac const deg}
With fixed degree, variable dimension, the function $\frac{1}{\sqrt{n}} \degree ^ \dimension \alpha_{\degree} ^ \dimension$ is a tight asymptotic bound for $\monomialsleqdeg{\dMac},\monomialsleqdeg{\dMac-1}$ and $\monomialsdeg{\dMac}$.
\end{lemma}
\begin{proof}
It is straightforward to verify that $\monomialsleqdeg{\dMac-1}\sim \frac{\degree-1}{\degree}\monomialsleqdeg{\dMac}$ and 
$\monomialsdeg{\dMac}\sim\frac{1}{\degree}\monomialsleqdeg{\dMac}$ as $\dimension\rightarrow\infty$. Using Stirling's approximation, it can be shown that 
$$\monomialsleqdeg{\dMac-1} 
\sim \left(\sqrt{\frac{2\pi\degree}{\degree-1}}\right) \frac{1}{\sqrt{n}} \degree ^ \dimension \alpha_{\degree} ^ \dimension.$$ 
The result follows.
\end{proof}

\begin{lemma}
\label{appendix: polysinmac deg dMac const deg}
With fixed degree, variable dimension, the function $\sqrt{n} \degree ^ \dimension \alpha_{\degree} ^ \dimension$ is a tight asymptotic bound for $\polysdeg{\dMac}$ and $\polysinmac{\dMac}$.
\end{lemma}
\begin{proof}
Since $\polysdeg{\dMac} = \dimension\ \monomialsleqdeg{\dMac - \degree}$, the function $\sqrt{n} \degree ^ \dimension \alpha_{\degree} ^ \dimension$ is a tight asymptotic bound for $\polysdeg{\dMac}$ if  $\monomialsleqdeg{\dMac - \degree}\sim C\monomialsleqdeg{\dMac-1}$ as $\dimension\rightarrow\infty$ for some $C>0$. It is straightforward to verify that
$$\dimensionlim \frac{\monomialsleqdeg{\dMac - \degree}}{\monomialsleqdeg{\dMac-1}} 
= \left(1-\frac{1}{\degree}\right)^\degree$$
and so the result holds for $\polysdeg{\dMac}$.

To prove the bound for $\polysinmac{\dMac}$, it suffices to show that $\polysdeg{\dMac}$ is a tight asymptotic bound for $\polysinmac{\dMac}$. Clearly $\polysinmac{\dMac} \geq \polysdeg{\dMac}$ because $\polysinmac{\dMac} = \sum_{k=\degree}^{\dMac}\polysdeg{k}$. If there is some $\theta > 1$ such that $\degree \leq k \leq \dMac$ implies $\polysdeg{k} \geq \theta \polysdeg{k-1}$, then 
$$\polysinmac{\dMac} = \sum_{k=\degree}^{\dMac}\polysdeg{k} \leq \sum_{k=0}^{\dMac-\degree} \frac{\polysdeg{\dMac}}{\theta^k} < \polysdeg{\dMac} \sum_{k=0}^{\infty}\frac{1}{\theta^k} = \polysdeg{\dMac} \frac{\theta}{\theta-1}$$
and we are done.

Now, let $\theta = 1 + \frac{1}{\degree-1}$. We see that 
\begin{align*}
\frac{\polysdeg{k}}{\polysdeg{k-1}} &= \frac{\binom{\dimension+k-\degree}{\dimension}}{\binom{\dimension+k-\degree-1}{\dimension}} \\ 
&=
1 + \frac{\dimension}{k-\degree} \\ 
&\geq
1 + \frac{\dimension}{\dMac-\degree} \\ 
&= 
1 + \frac{\dimension}{(\dimension-1)(\degree-1)} \\ 
&\geq
\theta.
\end{align*}
Multiplying both sides by $T_{k - 1}$ yields the desired result.

\end{proof}

\begin{lemma}
\label{appendix: degree-by-degree final step fixed deg}
With fixed degree, variable dimension, a tight asymptotic bound of the degree-by-degree construction is the same as the tight asymptotic bound of the final step.
\end{lemma}
\begin{proof}
Let the complexity of each step of the construction be $\DBDStep{k}$. The full complexity is $\sum_{k=\degree}^{\dMac}\DBDStep{k}$. This is clearly bounded below by $\DBDStep{\dMac}$, so it suffices to show $\sum_{k=\degree}^{\dMac}\DBDStep{k} < C\DBDStep{\dMac}$ for some $C$. Following the same reasoning as Lemma~\ref{appendix: polysinmac deg dMac const deg}, it suffices to show that $\DBDStep{k} > \theta \DBDStep{k-1}$ for some $\theta > 1$. From lemma~\ref{appendix: polysinmac deg dMac const deg} we have that $\polysdeg{k} > \theta_1 \polysdeg{k-1}$ for $\theta_1 > 1$. And $\frac{\monomialsleqdeg{k+1}}{\monomialsleqdeg{k}} = \frac{\binom{\dimension+k+1}{\dimension}}{\binom{\dimension+k}{\dimension}} = \frac{\dimension + k + 1}{k+1} = 1 + \frac{\dimension}{k+1}$. 

Using $k < \dMac$, this is 
$$1 + \frac{\dimension}{k+1} \geq 1+\frac{\dimension}{\dimension\degree-\dimension+2} > 1+\frac{\dimension}{\dimension\degree-\dimension} = 1 + \frac{1}{\degree - 1}.$$

So letting $\theta_2 = 1 + \frac{1}{\degree - 1}, \monomialsleqdeg{k} > \theta_2 \monomialsleqdeg{k-1}.$
Each of $\monomialsleqdeg{k}$, $\monomialsdeg{k}$, $\polysdeg{k}$, $\polysinmac{k}$ is increasing, and steps 1 and 2 of the construction contain a $\polysdeg{k}$ and step 3 contains a $\monomialsleqdeg{k}$, so letting $\theta = \min(\theta_1, \theta_2)$, we get that $\DBDStep{k} > \theta \DBDStep{k-1}$.
\end{proof}

\clearpage
\bibliographystyle{alpha}
\bibliography{Jarvis}

\begin{thebibliography}{TMVB18}

\bibitem[BC13]{cucker}
Peter B\"{u}rgisser and Felipe Cucker.
\newblock {\em Condition}, volume 349 of {\em Grundlehren der Mathematischen
  Wissenschaften [Fundamental Principles of Mathematical Sciences]}.
\newblock Springer, Heidelberg, 2013.
\newblock The geometry of numerical algorithms.

\bibitem[BCS10]{BCS}
Peter Brgisser, Michael Clausen, and Mohammad~A. Shokrollahi.
\newblock {\em Algebraic Complexity Theory}.
\newblock Springer Publishing Company, Incorporated, 1st edition, 2010.

\bibitem[Boy14]{BoydBook}
John~P. Boyd.
\newblock {\em Solving transcendental equations}.
\newblock Society for Industrial and Applied Mathematics, Philadelphia, PA,
  2014.
\newblock The Chebyshev polynomial proxy and other numerical rootfinders,
  perturbation series, and oracles.

\bibitem[CLO98]{Cox2}
David Cox, John Little, and Donal O'Shea.
\newblock {\em Using algebraic geometry}.
\newblock Graduate Texts in Mathematics, 185. Springer-Verlag, New York, 1998.

\bibitem[CO05]{UAG}
John Cox, David A. aofd~Little and Donal O'Shea.
\newblock {\em Using algebraic geometry}, volume 185 of {\em Graduate Texts in
  Mathematics}.
\newblock Springer, New York, second edition, 2005.

\bibitem[Eis95]{Eisenbud}
David Eisenbud.
\newblock {\em Commutative algebra}, volume 150 of {\em Graduate Texts in
  Mathematics}.
\newblock Springer-Verlag, New York, 1995.
\newblock With a view toward algebraic geometry.

\bibitem[GVL13]{Golub}
Gene~H. Golub and Charles~F. Van~Loan.
\newblock {\em Matrix Computations (4th Ed.)}.
\newblock Johns Hopkins University Press, Baltimore, MD, USA, 2013.

\bibitem[Kre14]{Kreuzer}
Martin Kreuzer.
\newblock {\em Computation of Approximate Border Bases and Applications}.
\newblock PhD thesis, Universit{\"a}t Passau, 2014.

\bibitem[Lat06]{Lathauwer}
Lieven Lathauwer.
\newblock A link between the canonical decomposition in multilinear algebra and
  simultaneous matrix diagonalization.
\newblock {\em SIAM J. Matrix Analysis Applications}, 28:642--666, 01 2006.

\bibitem[Mou07]{Mourrain}
Bernard Mourrain.
\newblock {Pythagore's Dilemma, Symbolic-Numeric Computation, and the Border
  Basis Method}.
\newblock In Dongming Wang and Lihong Zhi, editors, {\em {Symbolic-Numeric
  Computation}}, Trends in Mathematics, pages 223--243. {Birkhauser}, 2007.

\bibitem[MR94]{Macaulay}
F.S. Macaulay and P.L. Roberts.
\newblock {\em The Algebraic Theory of Modular Systems}.
\newblock Cambridge Mathematical Library. Cambridge University Press, 1994.

\bibitem[MT01]{Moller}
H.~Michael M\"oller and Ralf Tenberg.
\newblock Multivariate polynomial system solving using intersections of
  eigenspaces.
\newblock {\em Journal of Symbolic Computation}, 32:513--531, 11 2001.

\bibitem[MTV21]{Telen4}
Bernard Mourrain, Simon Telen, and Marc {Van Barel}.
\newblock Truncated normal forms for solving polynomial systems: Generalized
  and efficient algorithms.
\newblock {\em Journal of Symbolic Computation}, 102:63 -- 85, 2021.

\bibitem[NT16]{Noferini}
Vanni Noferini and Alex Townsend.
\newblock Numerical instability of resultant methods for multidimensional
  rootfinding.
\newblock {\em SIAM Journal on Numerical Analysis}, 54(2):719, 2016.

\bibitem[SK07]{Sasaki}
Tateaki Sasaki and Fujio Kako.
\newblock Computing floating-point gr\"{o}bner bases stably.
\newblock In {\em Proceedings of the 2007 International Workshop on
  Symbolic-numeric Computation}, SNC '07, pages 180--189, New York, NY, USA,
  2007. ACM.

\bibitem[Ste96]{Stetter}
Hans~J. Stetter.
\newblock Matrix eigenproblems are at the heart of polynomial system solving.
\newblock {\em SIGSAM Bull.}, 30(4):22--25, December 1996.

\bibitem[Ste04]{StetterBook}
Hans~J Stetter.
\newblock {\em Numerical polynomial algebra}, volume~85.
\newblock Siam, 2004.

\bibitem[TB97]{trefethen97}
Lloyd~N. Trefethen and David Bau.
\newblock {\em Numerical Linear Algebra}.
\newblock SIAM, 1997.

\bibitem[TMVB18]{Telen2}
Simon Telen, Bernard Mourrain, and Marc Van~Barel.
\newblock Solving polynomial systems via truncated normal forms.
\newblock {\em SIAM J. Matrix Anal. Appl.}, 39(3):1421--1447, 2018.

\bibitem[Tow15]{chebsuite}
Alex Townsend.
\newblock Chebfun2 root finding tests, 2015.

\bibitem[TVB18]{Telen}
Simon Telen and Marc Van~Barel.
\newblock A stabilized normal form algorithm for generic systems of polynomial
  equations.
\newblock {\em J. Comput. Appl. Math.}, 342:119--132, 2018.

\bibitem[VL87]{VanLoan}
Charles Van~Loan.
\newblock On estimating the condition of eigenvalues and eigenvectors.
\newblock {\em Linear Algebra Appl.}, 88/89:715--732, 1987.

\end{thebibliography}

\end{document}